\documentclass[12pt,notitlepage]{amsart}
\usepackage{latexsym,amsfonts,amssymb,amsmath,amsthm}
\usepackage{graphicx}
\usepackage{color}
\usepackage{mathrsfs}
\usepackage{amsmath, amsthm, amsfonts, amssymb, cite}
\usepackage{wrapfig}
\usepackage{stackrel}
\usepackage{xcolor}
\usepackage{float}
\usepackage{enumitem}
\usepackage{mathtools}
\usepackage{multicol}
\usepackage[normalem]{ulem}
\usepackage{cancel}

\usepackage[inner=1.0in,outer=1.0in,bottom=1.0in, top=1.0in]{geometry}

\newcommand{\mz}{\ensuremath{\mathbb Z}}
\newcommand{\mr}{\ensuremath{\mathbb R}}

\newcommand{\mc}{\ensuremath{\mathbb C}}

\newcommand{\gammamodh}{\ensuremath{\Gamma\backslash\mathbb H}}

\newcommand{\shortmod}{\ensuremath{\negthickspace \negthickspace \negthickspace \pmod}}

\newcommand{\half}{\ensuremath{ \frac{1}{2}}}

\newcommand{\intR}{\int_{-\infty}^{\infty}}

\newcommand{\sumstar}{\sideset{}{^*}\sum}

\newcommand{\cF}{\ensuremath{\mathcal{F}}}

\newenvironment{bothcases}
  {\left \lbrace \begin{aligned}}
  {\end{aligned}\right\rbrace}

\makeatletter
\@namedef{subjclassname@2020}{\textup{2020} Mathematics Subject Classification}
\makeatother

\theoremstyle{plain}	
	\newtheorem*{rema}{Remark}	
	\newtheorem{mytheo}{Theorem}[section]

	\newtheorem{mycoro}[mytheo]{Corollary}
     \newtheorem{mylemma}[mytheo]{Lemma}
	\newtheorem{mydefi}[mytheo]{Definition}
	
	\newtheorem{myremark}[mytheo]{Remark}

\theoremstyle{remark}

\numberwithin{equation}{section}
\begin{document}
\author{Rizwanur Khan}
\address{Department of Mathematics \\
 	 University of Mississippi \\
 	  University \\
 	  MS 38677}
 \email{rrkhan@olemiss.edu}

 \author{Matthew P. Young} 
 \address{Department of Mathematics \\
 	  Texas A\&M University \\
 	  College Station \\
 	  TX 77843-3368}
 \email{myoung@math.tamu.edu}
 

 \thanks{This material is based upon work supported by the National Science Foundation under agreement nos. DMS-2001183 (R.K.) and DMS-1702221 (M.Y.) and the Simons Foundation under award no. 630985 (R.K.).  Any opinions, findings and conclusions or recommendations expressed in this material are those of the authors and do not necessarily reflect the views of the National Science Foundation.}

 \begin{abstract} We establish sharp bounds for the second moment of symmetric-square $L$-functions attached to Hecke Maass cusp forms $u_j$ with spectral parameter $t_j$, where the second moment is a sum over $t_j$ in a short interval. At the central point $s=1/2$ of the $L$-function, our interval is smaller than previous known results. More specifically, for $|t_j|$ of size $T$, our interval is of size $T^{1/5}$, while the previous best was $T^{1/3}$ from work of Lam. A little higher up on the critical line, our second moment yields a subconvexity bound for the symmetric-square $L$-function.  More specifically, we get subconvexity at $s=1/2+it$ provided $|t_j|^{6/7+\delta}\le |t| \le (2-\delta)|t_j|$ for any fixed $\delta>0$. Since $|t|$ can be taken significantly smaller than $|t_j|$, this may be viewed as an approximation to the notorious subconvexity problem for the symmetric-square $L$-function in the spectral aspect at $s=1/2$.
\end{abstract}
%\today
\title{Moments and hybrid subconvexity for symmetric-square $L$-functions }
\maketitle

\section{Introduction} 

\subsection{Background}

The widely studied subconvexity problem for automorphic $L$-functions is completely resolved for degree $\le 2$. For uniform bounds, over arbitrary number fields, this is due to Michel and Venkatesh \cite{micven}; for superior quality bounds in various special cases, this is due to many authors, of which a small sample is \cite{ JM, blohar, bour, bhkm, PY}. The next frontier is degree $3$, but here the subconvexity problem remains a great challenge, save for a few spectacular successes. The first breakthrough is due to Xiaoqing Li \cite{li}, who established subconvexity for $L(f,1/2+it)$ on the critical line ($t$-aspect), where $f$ is a fixed self-dual Hecke-Maass cusp form for $SL_3(\mathbb{Z})$. This result was generalized by Munshi \cite{mun1}, by a very different method, to forms $f$ that are not necessarily self-dual.  Munshi \cite{mun2} also established subconvexity for twists $L(f\times \chi, 1/2)$ in the $p$-aspect, where $\chi$ is a primitive Dirichlet character of prime modulus $p$. Subconvexity in the spectral aspect of $f$ itself is much harder, and even more so when $f$ is self-dual due to a conductor-dropping phenomenon. Blomer and Buttcane \cite{blobut}, Kumar, Mallesham, and Singh \cite{KMS}, and Sharma \cite{sharma} have established subconvexity for $L(1/2, f)$ in the spectral aspect of $f$ in many cases, but excluding the self-dual forms.

A self-dual $GL_3$ Hecke-Maass cusp form is known to be a symmetric-square lift from $GL_2$ \cite{soud}. Let $u_j$ be a Hecke-Maass cusp form for the full modular group $SL_2(\mz)$, with Laplace eigenvalue $1/4 + t_j^2$. It is an outstanding open problem to prove subconvexity for the associated symmetric-square $L$-function $L(\mathrm{sym}^2 u_j, 1/2)$ in the $t_j$-aspect. Such a bound would represent major progress in the problem of obtaining a power-saving rate of decay in the Quantum Unique Ergodicity problem \cite{iwasar}. A related problem is that of establishing the Lindel\"of-on-average bound
\begin{align}
\label{secondmomentsum} \sum_{T\le t_j\le T+\Delta} |L(\mathrm{sym}^2 u_j, 1/2 +it)|^2 \ll \Delta T^{1+\epsilon}
\end{align}
where we assume throughout that $T^\epsilon \le \Delta\le T^{1-\epsilon}$, and we generally aim to take $\Delta$ as small as possible.  Such an estimate is interesting in its own right, and also yields by positivity a bound for each $L$-value in the sum. At the central point ($t=0$), if \eqref{secondmomentsum} can be established for $\Delta=T^\epsilon$, it would give the convexity bound for $L(\mathrm{sym}^2 u_j, 1/2)$; the hope would then be to insert an amplifier in order to prove subconvexity. While a second moment bound which implies convexity at the central point is known in the level aspect by the work of Iwaniec and Michel \cite{iwamic}, in the spectral aspect the problem is much more difficult. The best known result until now for \eqref{secondmomentsum} was $\Delta=T^{1/3+\epsilon}$ by Lam \cite{Lam}. (Lam's work actually involves symmetric-square $L$-functions attached to holomorphic Hecke eigenforms, but his method should apply equally well to Hecke-Maass forms.) Other works involving moments of symmetric square $L$-functions include \cite{Bl, K, J, KD, BF, Ba, N}.

\subsection{Main results}
One of the main results of this paper is an approximate version of the subconvexity bound for $L(\mathrm{sym}^2 u_j, 1/2)$. Namely, we establish subconvexity for $L(\mathrm{sym}^2 u_j, 1/2+it)$ for $t$ small, but not too small, compared to $2t_j$. This hybrid bound (stated precisely below) seems to be the first subconvexity bound for symmetric-square $L$-functions in which the dominant aspect is  the spectral parameter $t_j$. For comparison, note that bookkeeping the proofs of Li \cite{li} or Munshi \cite{mun1} would yield hybrid subconvexity bounds for $t_j$ (very) small compared to $t$. Our method also yields a hybrid subconvexity bound for $L(\mathrm{sym}^2 u_j, 1/2+it)$ when $t$ is larger (but not too much larger) than $2t_j$, but for simplicity we refrain from making precise statements. We do not prove anything when $t$ is close to $2t_j$, for in this case the analytic conductor of the $L$-function drops. In fact it is then the same size as the analytic conductor at $t=0$, where the subconvexity problem is the hardest.

Our approach is to establish a sharp estimate for the second moment as in \eqref{secondmomentsum}, which is strong enough to yield subconvexity in certain ranges.  
 \begin{mytheo} 
\label{thm:mainthm}
 Let $0<\delta<2$ be fixed, and let $U,T,\Delta>1$ be such that
 \begin{equation}
 \label{eq:mainthmcondition}
\frac{T^{3/2+\delta}}{\Delta^{3/2}} \le U\le (2-\delta) T.
 \end{equation}
We have
\begin{equation}
\label{eq:mainthmbound}
 \sum_{T < t_j <T+\Delta} 
 |L(\mathrm{sym}^2 u_j, 1/2+iU)|^2
 \ll \Delta T^{1+\epsilon}.
\end{equation}
\end{mytheo}

%\begin{mytheo}
%Suppose that $T^{\varepsilon} \ll \Delta \ll T^{1-\varepsilon}$,
% \begin{equation}
% \label{eq:mainthmcondition}
%\Delta U^{2/3} \gg_{\delta} T^{1+\delta},
% \end{equation}
%and $U \leq (2-\delta) T$ for some $\delta>0$.  Then
%\begin{equation}
% \sum_{T < t_j <T+\Delta} 
% |L(\mathrm{sym}^2 u_j, 1/2+iU)|^2
% \ll T^{\varepsilon}(\Delta T + \frac{T^{1/2} U }{\Delta^{1/2}}).
%\end{equation}
%\end{mytheo}
\begin{mycoro}
Let $0<\delta<2$ be fixed. For $|t_j|^{6/7+\delta} \leq U \leq (2-\delta) |t_j|$, we have the hybrid subconvexity bound 
\begin{equation} 
\label{eq:hybridsubconvexity}
 L(\mathrm{sym}^2 u_j, 1/2+iU) \ll |t_j|^{1+\epsilon} U^{-1/3}.
\end{equation} 
\end{mycoro}
\begin{proof}
The bound follows by taking $\Delta = T^{1+\delta}U^{-2/3}$ in Theorem \ref{thm:mainthm} with $\delta$ chosen small enough. When $U\ge T^{6/7+\delta}$, this bound is subconvex. 
\end{proof}

Note that in Theorem \ref{thm:mainthm}, we are able to take $\Delta$ as small as $T^{1/3}$ at best. This requires $T \ll U \leq (2-\delta) T$ and for instance yields the subconvexity bound $L(\mathrm{sym}^2 u_j, 1/2+it_j)\ll |t_j|^{2/3+\epsilon}$.

We might also speculate that the lower bound in \eqref{eq:mainthmcondition} could plausibly be relaxed to $\Delta U \gg T^{1+\delta}$ (possibly with an additional term on the right hand side of \eqref{eq:mainthmbound}, as in \eqref{eq:finalboundpresimplified}) which would give subconvexity in the wider range  $ T^{2/3+\delta} \le U\leq(2-\delta)T$. For some reasoning on this, see the remark following \eqref{eq:use-of-assumption}.

For the central values we do not get subconvexity but we are able to improve the state of the art for the second moment. This is the other main result of this paper: we establish a Lindel\"of-on-average estimate for the second moment with $\Delta$ as small as $T^{1/5+\epsilon}$.
\begin{mytheo}
\label{thm:CentralValueBound}
 For $\Delta \ge T^{1/5+\varepsilon}$ and $0\leq U \ll T^{\varepsilon}$ we have
 \begin{equation}
 \label{eq:CentralValueBound}
 \sum_{T < t_j <T+\Delta} 
 |L(\mathrm{sym}^2 u_j, 1/2 +iU)|^2
 \ll \Delta T^{1+\varepsilon}.
 \end{equation}
\end{mytheo}
%An analogous bound holds for holomorphic modular forms with the spectral sum replaced by a sum over the weights $T < k <T + \Delta$.  The previous best result is due to Lam \cite{Lam}, who obtained Lindel\"{o}f on average for $\Delta \gg T^{1/3}$.  
It is a standing challenge to prove a Lindel\"{o}f-on-average
bound in \eqref{eq:CentralValueBound} with $\Delta = 1$.

Theorem \ref{thm:CentralValueBound} also has implications for the quantum variance problem.  To explain this, recall that Quantum Unique Ergodicity \cite{lind, sound} says that for any smooth, bounded function $\psi$ on $\gammamodh$, we have that $\langle |u_j|^2,  \psi\rangle \to \frac{3}{\pi}\langle 1,\psi \rangle$ as $t_j\to\infty$. By spectrally decomposing $\psi$, this is equivalent to demonstrating the decay of $\langle |u_j|^2, \varphi\rangle$ and $\langle |u_j|^2,  E_U\rangle$, where $\varphi$ is a fixed Hecke-Maass cusp form and $E_U=E(\cdot, \half+iU)$ is the standard Eisenstein series with $U$ fixed. The quantum variance problem is the problem of understanding the variance of these crucial quantitites.  More precisely, the quantum variance problem asks for non-trivial bounds on 
\begin{equation}
\label{eq:QuantumVarianceCuspidalComponent}
\sum_{T < t_j <T+\Delta} |\langle u_j^2, \varphi \rangle|^2,
\end{equation}
as well as the Eisenstein contribution
\begin{equation}
\label{eq:QuantumVarianceEisensteinComponent}
\sum_{T < t_j <T+\Delta} |\langle u_j^2, E_U \rangle|^2.
\end{equation}
Our Theorem \ref{thm:CentralValueBound} gives, by classical Rankin-Selberg theory, a sharp bound 
on \eqref{eq:QuantumVarianceEisensteinComponent}
for $\Delta\ge T^{1/5+\epsilon}$. 
  In turn, by Watson's formula \cite{W}, a sharp estimate for \eqref{eq:QuantumVarianceCuspidalComponent} boils down to establishing 
\begin{equation}
 \label{eq:CentralValueBoundCuspidalCase}
 \sum_{T < t_j <T+\Delta} 
 L(\mathrm{sym}^2 u_j \otimes \varphi,1/2)
 \ll \Delta T^{1+\varepsilon}.
 \end{equation}
It is plausible that the methods used to prove Theorem \ref{thm:CentralValueBound} should also generalize to show \eqref{eq:CentralValueBoundCuspidalCase} for $\Delta\ge T^{1/5+\epsilon}$, which would improve \cite{J}, but this requires a rigorous proof.  
 For quantum variance in the level aspect, see \cite{N}.

\subsection{Overview} 
We now give a rough sketch of our ideas for Theorems \ref{thm:mainthm} and \ref{thm:CentralValueBound}, both of which consider the second moment of the symmetric-square $L$-function. Let $h(t)$ be a smooth function supported essentially on $T <| t| <T+\Delta$, such as the one given in \eqref{eq:h-definition}. For $0\le U\le (2-\delta)T$, the analytic conductor of $L(\mathrm{sym}^2 u_j, 1/2+iU)$ is of size $T^2(U+1)$, so using an approximate functional equation, we have roughly 
\[
 \sum_{j\ge 1}  |L(\mathrm{sym}^2 u_j, 1/2+iU)|^2 h(t_j)=  \sum_{j\ge 1} \  \sum_{m,n \le T^{1+\epsilon}(U+1)^{1/2} } \frac{\lambda_j(m^2)\lambda_j(n^2)}{m^{1/2+iU}n^{1/2-iU}} h(t_j),
\]
which we need to show is bounded by $T^{1+\epsilon}\Delta$. Applying the Kuznetsov formula, the diagonal contribution is of size $O(T^{1+\epsilon}\Delta)$, while the off-diagonal contribution is roughly
\[
 \sum_{m,n \le T^{1+\epsilon}(U+1)^{1/2} } \frac{1}{m^{1/2+iU}n^{1/2-iU}} \sum_{c\ge 1} \frac{S(m^2,n^2,c)}{c} H\Big(\frac{4\pi mn}{c}\Big)
\]
for some transform $H$ of $h$, given in \eqref{Hdef}. We have by developing \eqref{eq:K+formula} that $H(x)$ is essentially supported on $x\ge T^{1-\epsilon}\Delta$ and roughly has the shape $H(x)=\frac{T\Delta}{x^{1/2}} e^{i(x - T^2/x)}$.  Thus in the generic ranges $m,n\sim T(U+1)^{1/2}$ and $c\sim \frac{mn}{T\Delta}$, writing $(n/m)^{i
U}=e(U\log (n/m)/2\pi)$ and not being very careful about factors of $\pi$ and such, the off-diagonal is
\begin{align}
\label{off-diag-sketch} \frac{\Delta^{3/2}}{U^{3/2}T^{3/2}}  \sum_{m,n \sim T(U+1)^{1/2} }  \ \sum_{c\sim T(U+1)/\Delta} S(m^2,n^2,c) e\Big( \frac{2 mn}{c}\Big)e\Big(-\frac{T^2c}{mn} +U \log (n/m) \Big).
\end{align}
The oscillatory factor $e( -\frac{T^2c}{nm} +U \log (n/ m))$ behaves differently according to whether $U$ is large or small.  When $U$ is large, the dominant phase is $U \log (n/ m)$, while when $U$ is small, the dominant phase is $-\frac{T^2c}{nm}  \sim - \frac{T}{\Delta}$.

Consider one extreme end of our problem: the case $U=T$ (covered by Theorem \ref{thm:mainthm}), so that the convexity bound is $T^{3/4+\epsilon}$. Since the diagonal after Kuznetsov is $O(T^{1+\epsilon}\Delta)$, the largest we can take $\Delta$ to establish subconvexity is $\Delta=T^{1/2-\delta}$ for some $\delta>0$. Thus for the off-diagonal, what we need to prove is roughly (specializing \eqref{off-diag-sketch} to $U= T, \Delta=T^{1/2}$ and retaining only the dominant phase)
\begin{align}
\label{off-diag-T}
 \frac{1}{T^{9/4}}   \sum_{m,n\sim T^{3/2} }  \  \sum_{c\sim T^{3/2} } S(m^2,n^2,c) e\Big( \frac{2mn}{c}\Big) e(T \log (n/m) )\ll T^{3/2}.
\end{align}
We split the $n$ and $m$ sums into residue classes modulo $c$ and apply Poisson summation to each. The off-diagonal then equals
\[
 \frac{1}{T^{9/4}}   \sum_{c\sim T^{3/2} } \ \sum_{k,\ell\in \mathbb{Z} } \frac{1}{c^2} T(k,\ell,c) I(k,\ell,c),
\]
where
\[
 I(k,\ell, c) = \iint e\Big(\frac{-kx-\ell y}{c}+T\log x-T\log y\Big)w\Big(\frac{x}{T^{3/2}},\frac{y}{T^{3/2}}\Big) dxdy
\]
for some smooth weight function $w$ which restricts support to $x\sim T^{3/2}, y\sim T^{3/2}$, and 
\[
T(k,\ell,c) =\sum_{a,b \bmod c} S(a^2, b^2, c)e\Big(\frac{2ab+ak+b\ell}{c}\Big).
\]
We compute this arithmetic sum in section \ref{section:misclemmas} and roughly get $T(k,\ell, c) =  c^{3/2} (\frac{k\ell}{c}) e(\frac{-k\ell}{4c} )$.
The integral is computed using stationary phase (see Sections \ref{section:oscillatoryintegrals} and \ref{poissonintegral}). We see that it is negligibly small unless $k,\ell \sim T$, in which case we get roughly $I(k,\ell,c)= T^2e(\frac{k\ell }{c}) (k/\ell)^{iT}$ (see Lemma \ref{lemma:IkellcAsymptoticForUBig} for the rigorous statement). Thus we need to show
\[
\frac{1}{T} \sum_{k,\ell\sim T} \  \sum_{c\sim T^{3/2}} \Big(\frac{k}{\ell}\Big)^{iT}  e\Big(\frac{ 3 k\ell}{4 c}\Big) \Big(\frac{k\ell}{c}\Big)\ll T^{3/2}.
\]
At this point we go beyond previous approaches to the second moment problem \cite{iwamic, Lam} by finding cancellation in the $c$ sum. We split the $c$ sum into arithmetic progresssions modulo $k\ell$ by quadratic reciprocity and apply Poisson summation,
getting that the off-diagonal equals 
\begin{equation}
\label{equation:Poisson-c-step}\frac{1}{T} \sum_{k,\ell\sim T} \Big(\frac{k}{\ell}\Big)^{iT}  \sum_{q\in\mathbb{Z}} \frac{1}{k\ell} \sum_{a\bmod k\ell} \Big(\frac{a}{k\ell}\Big)e\Big(  \frac{-aq}{k\ell}\Big) \int e\Big(\frac{ 3k\ell}{4x}+\frac{qx}{k\ell} \Big) w\Big(\frac{x}{T^{3/2}}\Big)dx.
\end{equation}
This Poisson summation step may be viewed as the key new ingredient in our paper. It leads to a simpler expression in two ways.  Firstly, an integration by parts argument shows that the $q$-sum can be restricted to $q \sim T$, which is significantly shorter than the earlier $c$-sum of length $T^{3/2}$.  A more elaborate stationary phase analysis of the integral shows that the integral is essentially independent of $k$ and $\ell$, which can be seen in rough form by the substitution $x\to xk\ell$ in \eqref{equation:Poisson-c-step}. The reader will not actually find an expression like \eqref{equation:Poisson-c-step} in the paper because we execute Poisson summation in $c$ in the language of Dirichlet series and functional equations. This allows us to more effectively deal with some of the more delicate features of this step. For example, see Remark \ref{remark-c}.

Evaluating the arithmetic sum and using stationary phase to compute the integral in \eqref{equation:Poisson-c-step}, we get that the off-diagonal equals
\begin{align}
\label{end-T} \frac{1}{T^{3/4}}\sum_{k,\ell\sim T} \  \sum_{q\sim T} e(\sqrt{q})\Big(\frac{q}{k\ell}\Big)\Big(\frac{k}{\ell}\Big)^{iT} = \frac{1}{T^{3/4}} \sum_{q\sim T} e(\sqrt{q}) \Big| \sum_{k\sim T} \Big(\frac{q}{k}\Big)k^{iT}\Big|^2.
\end{align}
Finally, applying Heath-Brown's \cite{HB} large sieve for quadratic characters, we get that the off-diagonal is $O(T^{5/4+\epsilon})$, which is better than the required bound in \eqref{off-diag-T}.

Now consider Theorem \ref{thm:CentralValueBound}, which deals with the other extreme end of our problem where $U$ is small. The treatment of  this follows the same plan as sketched above for large values of $U$, but the details are changed a bit because the oscillatory factor in \eqref{off-diag-sketch} behaves differently.  Consider the case $U=0$ (the central point) and $\Delta=T^{1/5}$, which is the best we can do in Theorem \ref{thm:CentralValueBound}. In the end, instead of \eqref{end-T}, one arrives roughly at an expression of the form 
\begin{equation}
\label{eq:sketchCentralValueBound}
 \sum_{q \sim T^{6/5}} e(T^{1/2}q^{1/4}) \Big|\sum_{k \sim T^{3/5}} \frac{ (\frac{k}{q})  }{\sqrt{k}} \Big|^2.
\end{equation}
Again, Heath-Brown's quadratic large sieve is the end-game, giving a bound of $\Delta T^{1+\epsilon}=T^{6/5+\epsilon}$. It is a curious difference that the $q$-sum in \eqref{eq:sketchCentralValueBound} is now actually \emph{longer} than the  $c$-sum from which it arose via Poisson summation, in contrast to the situation with $U = T$ presented earlier.  However, the gain is that the variables $q$ and $k$ become separated in the exponential phase factor (indeed, $k$ is completely removed from the phase in \eqref{eq:sketchCentralValueBound}).

\subsection{Notational Conventions}
\label{section:notation}
Throughout, we will follow the epsilon convention, in which $\epsilon$ always denotes an arbitrarily small positive constant, but not necessarily the same one from one occurrence to another. As usual, we will write $e(x)=e^{2\pi i x}$, and $e_c(x) = e(x/c)$.
For $n$ a positive odd integer, we let $\chi_n(m) = (\frac{m}{n})$ denote the Jacobi symbol.
If $s$ is complex, an expression of the form $O(p^{-s})$ should be interpreted to mean $O(p^{-\mathrm{Re}(s)})$.  This abuse of notation will only be used on occasion with Euler products.
We may also write $O(p^{- \min(s,u)})$ in place of $O(p^{-\min(\mathrm{Re}(s), \mathrm{Re}(u))})$.

Upper bounds in terms of the size of $U$ are usually expressed, since $U$ may be $0$, in terms of $1+U$. However to save clutter, such upper bounds will be written in terms of $U$ only. This is justified at the start of section \ref{section:startofproof}.

{\bf Acknowledgement.} We are grateful to the anonymous referee for an exceptionally thorough and helpful review.

\section{Automorphic forms}
\subsection{Symmetric-square $L$-functions}
Let $u_j$ be a Hecke-Maass cusp form for the modular group $SL_2(\mz)$ with Laplace eigenvalue $1/4 + t_j^2$, and $n$-th Hecke eigenvalue $\lambda_j(n)$. It has an associated symmetric-square $L$-function defined by
%\begin{equation}
% L(\mathrm{sym}^2 u_j, s) = \sum_{m, n \geq 1} \frac{\lambda_j(n^2)}{(m^2 n)^s}, \qquad \mathrm{Re}(s) > 1.
%\end{equation}
$L(\mathrm{sym}^2 u_j, s) = \sum_{n \geq 1} \lambda_{\mathrm{sym}^2 u_j}(n) n^{-s}$, with $\lambda_{\mathrm{sym}^2 u_j}(n) = \sum_{a^2 b = n} \lambda_j(b^2)$. 
Let $\Gamma_{\mr}(s) = \pi^{-s/2} \Gamma(s/2)$ and 
%\begin{equation}
$\gamma(\mathrm{sym}^2 u_j, s) = \Gamma_{\mr}(s) \Gamma_{\mr}(s+2it_j) \Gamma_{\mr}(s-2it_j)$.
%\end{equation}
Then $L(\mathrm{sym}^2 u_j, s)$ has an analytic continuation to $\mc$ and satisfies the functional equation $\gamma(\mathrm{sym}^2 u_j, s) L(\mathrm{sym}^2 u_j, s) = \gamma(\mathrm{sym}^2 u_j, 1-s) L(\mathrm{sym}^2 u_j, 1-s)$, where the notation for $\gamma(f,s)$ agrees with \cite[Chapter 5]{IK}.
In particular, the analytic conductor of $L(\mathrm{sym}^2 u_j, 1/2 + it)$ equals
\begin{equation}
(1+|t|)(1+|t+2t_j|)(1+|2t_j-t|).
\end{equation}

\subsection{The Kuznetsov formula}
Let $h(z)$ be an even, holomorphic function on $|\Im(z)|<\frac12 +\delta$, with decay $|h(z)|\ll (1+|z|)^{2-\delta}$, for some $\delta>0$. Let $\{u_j: j\ge 1\}$ denote an orthonormal basis of Maass cusp forms of level $q$ with Laplace eigenvalue $\frac14+t_j^2$ and Fourier expansion
\[
u_j(z)=y^{\frac12} \sum_{n\neq 0} \rho_j(n) K_{it_j}(2\pi |n| y)e(nx),
\]
where $z=x+iy$ and $K_{it_j}$ is the $K$-Bessel function.
At each inequivalent cusp $\mathfrak{a}$ of $\Gamma_0(q)$, let $E_\mathfrak{a}(\cdot, \frac12+it)$ be the associated Eisenstein series with Fourier expansion
\[
E_\mathfrak{a}(z, \tfrac12+it)=\delta_{\mathfrak{a}=\infty} y^{\frac12+it} + \varphi_{\mathfrak{a}}(\tfrac12+it)y^{\frac12-it}+y^{\frac12} \sum_{n\neq 0} \tau_{\mathfrak{a}}(n,t) K_{it}(2\pi |n| y)e(nx),
\]
where $\varphi_{\mathfrak{a}}(s)$ is meromorphic on $\mc$. These expansions may be found in \cite[(16.19),(16.22)]{IK}.

\begin{mylemma}[Kuznetsov's formula{\cite[Theorem 16.3]{IK}}]
\label{lemma:Kuztrace}
 For any $n,m>0$ we have
\begin{multline*}
\sum_{j\ge 1} \rho_j(n)\overline{\rho}_j(m)\frac{h(t_j)}{\cosh(\pi t_j)}+\sum_{\mathfrak{a}}\frac{1}{4\pi}\int_{-\infty}^{\infty} \tau_\mathfrak{a}(n,t)\overline{\tau}_\mathfrak{a}(m,t) \frac{h(t)dt}{\cosh(\pi t)}\\
=\delta_{(n=m)}\int_{-\infty}^\infty  h(t) t \tanh(\pi t) \frac{dt}{\pi^2} + \frac{i}{\pi} \sum_{c\equiv 0 \bmod q}\frac{S(n,m,c)}{c}  \int_{-\infty}^\infty J\Big(\frac{4\pi\sqrt{nm}}{c},t\Big)h(t) t\tanh(\pi t)dt,
\end{multline*}
where $\displaystyle J(x,t) =\frac{J_{2it}(x) - J_{-2it}(x)}{\sinh(\pi t)}$.

\end{mylemma}

Later, we will need to use the Kuznetsov formula for level $2^4$. We will choose our orthonormal basis to include the level 1 Hecke-Maass forms, for which we may write 
\[
\rho_j(n)\overline{\rho}_j(m)\frac{h(t_j)}{\cosh(\pi t_j)}=\lambda_j(n)\overline{\lambda}_j(m) \frac{h(t_j) |p_j(1)|^2}{\cosh(\pi t_j)},
\]
and note that $t_j^{-\epsilon}\ll \frac{|\rho_j(1)|^2}{\cosh(\pi t_j)}\ll t_j^\epsilon$ by \cite[(30)]{HM} together with the fact that $L^2$-normalization in $\Gamma_0(2^4)$ and $\Gamma_0(1)$ is the same up to a constant factor.

\section{The quadratic large sieve}
We will have need of Heath-Brown's large sieve inequality for quadratic characters:
\begin{mytheo}[Heath-Brown \cite{HB}]
Let $M, N \gg 1$.  Then
\begin{equation}
\sumstar_{m \leq M}
\Big|
\sumstar_{n \leq N} 
a_n \Big(\frac{n}{m}\Big)\Big|^2 \ll (M+N)(MN)^{\varepsilon} \sum_{n \leq N} |a_n|^2,
\end{equation}
where the sums are restricted to odd square-free integers.
\end{mytheo}
We will need a corollary of Heath-Brown's result, namely
\begin{equation}
\sum_{m \leq M} 
|L(1/2 + it, \chi_{m})|^2 \ll (M+ \sqrt{M(1+|t|)})^{} (M(1+|t|))^{\varepsilon},
\end{equation}
This follows from an approximate functional equation, and a simple observation that the square parts of the inner and outer variables are harmless.
Similarly, we obtain
\begin{equation}
\label{eq:HBsecondmomentwithweights}
\sum_{m \leq M} m^{-1/2}
|L(1/2 + it, \chi_{m})|^2 \ll 
\Big(M^{1/2} + (1+|t|)^{1/2}
\Big)
 (M(1+|t|))^{\varepsilon}.
\end{equation}

\section{Oscillatory integrals}
\label{section:oscillatoryintegrals}
Throughout this paper we will make extensive use of estimates for oscillatory integrals.  We will largely rely on the results of \cite{KPY} (built on \cite{BKY}) which uses the language of families of inert functions.  This language gives a concise way to track bounds on derivatives of weight functions.  It also has the pleasant property that, loosely speaking, the class of inert functions is closed under application of the stationary phase method (the precise statement is in Lemma \ref{lemma:exponentialintegralStatPhase} below).  We refer the reader to \cite{KPY} for a more thorough discussion, including examples of applying stationary phase using this language.

Let $\cF$ be an index set and $X=X_T: \cF \to \mr_{\geq 1}$ be a function of $T \in \cF$.
\begin{mydefi}
\label{inert}
A family $\{w_T\}_{T\in \cF}$ of smooth  
functions supported on a product of dyadic intervals in $\mr_{>0}^d$ is called $X$-inert if for each $j=(j_1,\ldots,j_d) \in \mz_{\geq 0}^d$ we have 
\begin{equation}
\label{inertproperty} C_{\mathcal{F}}(j_1,\ldots,j_d) 
:= \sup_{T \in \cF} \sup_{(x_1, \ldots, x_d) \in \mr_{>0}^d} X_T^{-j_1- \cdots -j_d}\left| x_{1}^{j_1} \cdots x_{d}^{j_d} w_T^{(j_1,\ldots,j_d)}(x_1,\ldots,x_d) \right| < \infty.
\end{equation}
\end{mydefi}
As an abuse, we might say that a single function is $1$-inert (or simply inert) by which we should mean that it is a member of a family of $1$-inert functions .

\begin{mylemma}[Integration by parts bound  \cite{BKY}]
\label{lemma:exponentialintegral}
 Suppose that $w = w_T(t)$ is a family of $X$-inert functions, 
 with compact support on $[Z, 2Z]$, so that for all $j=0,1,\dots$ we have the bound
$w^{(j)}(t) \ll (Z/X)^{-j}$.  Also suppose that $\phi$ is smooth and satisfies 
for $j=2,3, \dots$
$\phi^{(j)}(t) \ll \frac{Y}{Z^j}$ for some $R \geq 1$ with $Y/X \geq R$ and all $t$ in the support of $w$.  Let
\begin{equation*}
 I = \intR w(t) e^{i \phi(t)} dt.
\end{equation*}
 If $|\phi'(t)| \gg \frac{Y}{Z}$ for all $t$ in the support of $w$, then $I \ll_A Z R^{-A}$ for $A$ arbitrarily large.
\end{mylemma}

\begin{mylemma}[Stationary phase, \cite{BKY} \cite{KPY}]
\label{lemma:exponentialintegralStatPhase}
 Suppose $w_T$ is $X$-inert in $t_1, \dots t_d$, supported on $t_1 \asymp Z$ and $t_i \asymp X_i$ for $i=2,\dots, d$.  Suppose that on the support of $w_T$, $\phi = \phi_T$ satisfies 
\begin{equation}
\label{eq:phiderivatives}
 \frac{\partial^{a_1 + a_2 + \dots + a_d}}{\partial t_1^{a_1} \dots \partial t_d^{a_d}} \phi(t_1, t_2, \dots, t_d) \ll_{C_\mathcal{F}} \frac{Y}{Z^{a_1}} \frac{1}{X_2^{a_2} \dots X_d^{a_d}},
\end{equation}
for all $a_1, \ldots, a_d \in \mathbb{N}$ with $a_1\ge 1$.
Suppose $\phi''(t_1, t_2, \dots, t_d) \gg \frac{Y}{Z^2}$, (here and below, $\phi'$  and $\phi''$ denote the derivative with respect to $t_1$) for all $t_1, t_2, \dots, t_d$ in the support of $w_T$, and 
for each $t_2, \dots, t_d$ in the support of $\phi$
there exists $t_0 \asymp Z$ such that $\phi'(t_0, t_2, \dots, t_d) = 0$. 
Suppose that $Y/X^2 \geq R$ for some $R\geq 1$.  Then 
%$Y \gg X^2 q^\varepsilon$ with $q \geq 1$, 
\begin{equation}
\label{eq:IasymptoticMainThm}
I = \int_{\mathbb{R}} e^{i \phi(t_1, \dots, t_d)} w_T(t_1, \dots, t_d) dt_1 =  \frac{Z}{\sqrt{Y}} e^{i \phi(t_0, t_2, \dots, t_d)} W_T(t_2, \dots, t_d)
 + O_{A}(ZR^{-A}),
\end{equation}
for some $X$-inert family of functions $W_T$, and where $A>0$ may be taken to be arbitrarily large.
The implied constant in \eqref{eq:IasymptoticMainThm} depends only on $A$ and on $C_{\mathcal{F}}$ defined in \eqref{inertproperty}.
\end{mylemma}
The fact that $W_T$ is inert with respect to the same variables as $w_T$ (with the exception of $t_1$, of course) is highly convenient.  In practice, we may often temporarily suppress certain variables from the notation.  This is justified provided that the functions satisfy the inertness condition in terms of these variables. 
We also remark that if $d=1$, then $W_T(t_{2}, \dots t_d)$ is a constant.

The following remark will be helpful for using Lemma \ref{lemma:exponentialintegralStatPhase} in an iterative fashion. First note that $t_0$ is the unique function of $t_2, \dots, t_d$ which solves $\phi'(t_1, \dots, t_d) = 0$ when viewed as an equation in $t_1$. In other words, $t_0$ is defined implicitly by $\phi'(t_0, \dots, t_d) = 0$.  In practice it might be an unwelcome task to explicitly solve for $t_0$, and the following discussion will aid in avoiding this issue.
Let
\begin{equation}
 \Phi(t_2, \dots, t_d) = \phi(t_0, t_2, \dots, t_d),
\end{equation}
so by the chain rule,
\begin{equation}
 \frac{\partial}{\partial t_j} \Phi(t_2, \dots, t_d) = 
 \phi'(t_0, t_2, \dots, t_d) \frac{\partial t_0}{\partial t_j}
 + \frac{\partial }{\partial t_j} \phi(t_0, \dots, t_j) = 
 \frac{\partial}{\partial t_j} \phi(t_0, \dots, t_j),
 \end{equation}
and so on for higher derivatives.  Hence the derivatives of $\Phi$ have the same bounds as those on $\phi$ (supposing uniformity with respect to the first variable $t_1$).  

As a simple yet useful consequence of this, if $\phi$ satisfies \eqref{eq:phiderivatives} (with $Z$ replaced by $X_1$, say) as well as
$\frac{\partial^2}{\partial t_j^2} \phi(t_1, \dots, t_d) \gg \frac{Y}{X_j^2}  \geq R \geq 1$ for $j=1,2,\dots, k$, uniformly for all $t_1, \dots, t_d$ in the support of $w_T$, then 
\begin{multline}
\int_{\mr^k} e^{i\phi(t_1, \dots, t_d)} w_T(t_1, \dots, t_d) dt_1 \dots dt_k = 
\frac{X_1 \dots X_k}{Y^{k/2}} 
e^{i \phi({\bf v_0} ; t_{k+1}, \dots, t_d)} W_T(t_{k+1}, \dots, t_d)
\\
+ O\Big(\frac{X_1 \dots X_k}{R^A}\Big),
\end{multline}
where ${\bf v_0} \in \mr^k$ is the solution to $\nabla \phi({\bf v_0};t_{k+1}, \dots, t_d) = 0$, where the derivative is with respect to the first $k$ variables only (i.e. the first $k$ entries of $\nabla\phi$ are zero).
Here we have trivially integrated each error term over any remaining variables of integration; the arbitrarily large power of $R$ savings nicely allows for this crude treatment of the error terms.

The following is an archimedean analog of the well-known change of basis formula from additive to multiplicative characters (compare with \cite[(3.11)]{IK})
\begin{mylemma}
\label{lemma:mellinIntegralOfAdditiveCharacter}
Suppose that $w_T$ is $1$-inert, supported on $x \asymp X$ where $X\gg 1$.  Then
\begin{equation}
\label{eq:mellinIntegralOfAdditiveCharacter}
e^{-ix} w_T(x) = X^{-1/2} \int_{-t \asymp X} v(t) x^{it} dt + O(X^{-100}),
\end{equation}
where $v(t) = v_{X}(t)$ is some smooth function satisfying $v(t) \ll 1$.  Moreover, $v(t) = e^{-it \log(|t|/e)} W(t)$ for some $1$-inert function $W$ supported on $-t \asymp X$.
\end{mylemma}
%Remark.  To treat a similar case where $x$ is negative, one may simply conjugate \eqref{eq:mellinIntegralOfAdditiveCharacter}.
\begin{proof}
Let $f(x) = e^{-ix} w_T(x)$.  By Mellin inversion,
\begin{equation}
f(x) = \int_{(\sigma)} \frac{\widetilde{f}(-s)}{2\pi i} x^s ds, 
\quad \text{where} \quad
\widetilde{f}(-s) = \int_0^{\infty} e^{-ix} x^{-s} w_T(x) \frac{dx}{x}.
\end{equation}
Take $\sigma = 0$, so $s=it$.
Lemma \ref{lemma:exponentialintegral} implies that $\widetilde{f}(-it)$ is very small outside of the interval $-t \asymp X$.  For $-t \asymp X$, Lemma \ref{lemma:exponentialintegralStatPhase} gives that
\begin{equation}
\widetilde{f}(-it) = X^{-1/2} e^{-it \log(|t|/e)} W(t) + O(X^{-200}),
\end{equation}
where $W$ is a $1$-inert function supported on $-t \asymp X$.
\end{proof}
For later use, we record some simple consequences of the previous lemmas.

\begin{mylemma}
\label{lemma:SomeIntegralWithv}
Let $v(t) = e^{-it \log(|t|/e)} W(t)$ for some $1$-inert function $W$ supported on $-t \asymp X$ with $X\gg 1$. Let $\gamma(s) = \pi^{-s/2} \Gamma(\frac{s+\kappa}{2})$ for $\kappa \in \{0, 1 \}$. Let $D(s) = \sum_{n=1}^{\infty} a_n n^{-s}$ be a Dirichlet series absolutely convergent for $\mathrm{Re}(s) = 0$ with $\max_{t \in \mr} |D(it)| \leq A$ for some $A\ge 0$. Let $c_1,c_2,c_3$ be some real numbers (which may vary with $X$) with $0 \le c_1\ll 1$ and $ |c_2|X^3+|c_3| \ll X^{1-\delta}$ for some $\delta>0$. For any $Y >  0$ we have
\begin{align}
\label{equation:vintbound1} 
X^{-1/2}  \intR  v(t)  e^{-c_1 it\log|t|+c_2it^3} Y^{it} D(it) dt \ll_{v,A} 1
\end{align}
and 
\begin{align}
\label{equation:vintbound2} X^{-1/2}  \intR v(t) e^{-c_1i t\log|t|+c_2it^3} \frac{\gamma(1/2-i(t+c_3))}{\gamma(1/2+i(t+c_3))}  Y^{it} D(it) dt \ll_{v,A} 1.
\end{align}
The bounds depend only on $v$ and $A$. 
\proof
Expanding out the Dirichlet series, and exchanging summation and integration, it suffices to prove the result with $D(s)=1$.  
We first consider \eqref{equation:vintbound1}, which is an oscillatory integral with phase 
$$\phi(t) = -(1+c_1) t \log{|t|} +  t \log(eY) + c_2 t^3.$$
Note that the leading phase points in the direction $-t\log|t|$. 
For $|t| \asymp X$ we have
$\phi'(t) = -(1+c_1) \log{|t|} + \log(Y) -c_1 + O(X^{-\delta})$.  Lemma \ref{lemma:exponentialintegral} shows that the left hand side of \eqref{equation:vintbound1} is very small unless $\log{Y} = (1+c_1) \log{X} + O(1)$, for a sufficiently large implied constant.  On the other hand, if $\log{Y} = (1+c_1) \log{X} + O(1)$, then $\phi'(t) = -(1+c_1) \log(|t|/X) - (1+c_1) \log{X} + \log(Y) +O(1)= O(1)$.  We may then
use Lemma \ref{lemma:exponentialintegralStatPhase} to show the claimed bound \eqref{equation:vintbound1}.

For the second bound \eqref{equation:vintbound2}, we first observe that by Stirling's formula we have we have 
\begin{equation*}
\frac{\gamma(1/2-i(t+c_3))}{\gamma(1/2+i(t+c_3))} = W(t) e^{- i (t+c_3) \log|t+c_3| + cit} + O(X^{-200}),
\end{equation*}
for some $1$-inert function $W$ and some $c \in \mr$.  With the phase of this gamma ratio pointing in the same direction as $-t\log|t|$, we can repeat the same argument as above to show square root cancellation.
\endproof
\end{mylemma}

We end this section with some heuristic motivation for the bound in \eqref{equation:vintbound2}, and how it is related to \eqref{equation:Poisson-c-step} from the sketch.  Let $w$ be a fixed inert function, $C \gg 1$ and $P:=A/C \gg 1$.  By Poisson summation, we have
\begin{equation}
\label{eq:throwawaysumcomment}
S := \sum_{c=1}^{\infty} e\Big(-\frac{A}{c}\Big) w(c/C) = \sum_q \intR e\Big(-\frac{A}{t} - qt \Big) w(t/C) dt.
\end{equation}
Integration by parts and stationary phase tells us that the sum is essentially supported on $q \asymp \frac{A}{C^2}$ in which case the integral is bounded by $\frac{C}{\sqrt{P}}$.  An alternative (and admittedly more roundabout!) way to accomplish this same goal is to use Lemma \ref{lemma:mellinIntegralOfAdditiveCharacter} with $x = 2 \pi \frac{A}{c}$, and the functional equation of the Riemann zeta function (shifting contours appropriately).  The dual sum will have a test function of the form on the left hand side of \eqref{equation:vintbound2} (with $c_3=0$ in fact), and the bound in \eqref{equation:vintbound2} is consistent with the simpler Fourier analysis presented in this paragraph above.  The reader may wonder, then, why we have proceeded in this more complicated fashion if the Fourier approach is simpler.  The answer is that the actual sums we encounter in this paper are arithmetically much more intricate than the simplified one presented in \eqref{eq:throwawaysumcomment}.  The Mellin transform approach is better-suited to handling the more complicated arithmetical features that are present in our problem, so on the whole, taking into account both the analytic and arithmetic aspects of the problem, the Mellin transform approach is simpler.

\section{Character sum evaluations}
\label{section:misclemmas}
We need the following elementary character sum calculations.  Define the Gauss sum
\begin{equation}
 \label{gausssumdef} G\Big(\frac{a}{c}\Big) = \sum_{x \shortmod{c}} e_c(ax^2).
\end{equation}
We need to evaluate $G(a/c)$.  It is well known (e.g. see \cite[(3.22), (3.38)]{IK}) that
\begin{equation}
\label{eq:GaussSumQuadraticOddModulusEvaluation}
 G\Big(\frac{a}{c}\Big) = \Big(\frac{a}{c}\Big) \epsilon_c \sqrt{c}, \qquad 
 \epsilon_c = 
 \begin{cases}
  1, \qquad &c \equiv 1 \pmod{4} \\
  i, \qquad &c \equiv 3 \pmod{4},
 \end{cases}
\end{equation}
provided $(2a,c) = 1$.
The case with $c$ even is treated as follows. Let $\delta\in\{0,1\}$ indicate the parity of the highest power of $2$ dividing $c$, as follows: if $2^{v_2}\| c$ then let
\begin{equation}
\label{delta-parity}
\delta \equiv v_2 \pmod 2.
\end{equation}
From the context, this should not be confused with usages where $\delta$ is a small positive constant or the $\delta(P)$ function which equals $1$ when a statement $P$ is true and $0$ otherwise.

\begin{mylemma}
\label{lemma:Gcalc}
Suppose $c=2^k c_o$ with $k \geq 2$, $c_o$ odd, and $\delta$ is as in \eqref{delta-parity}.
Suppose also $(a,c)=1$.  Then 
\begin{equation}
G\Big(\frac{a}{c}\Big) = \epsilon_{c_o} c^{1/2} \Big(\frac{a 2^{\delta}}{c_o}\Big) 
\begin{cases}
 1 + e_4(a c_o), \qquad \delta=0 \\
 2^{1/2} e_8(a c_o), \qquad \delta=1.
\end{cases}
.
\end{equation}
\end{mylemma}
\begin{proof}
First we note that if $c = c_1 c_2$ with $(c_1, c_2) = 1$, then
\begin{equation}
 G\Big(\frac{a}{c_1 c_2}\Big) = 
 G\Big(\frac{a c_2}{c_1 }\Big)
 G\Big(\frac{a c_1}{c_2 }\Big).
\end{equation}
Suppose that $c= 2^k$ with $k \geq 2$.  Let $j$ be an integer so that $2j \geq k$, and write $x = u + 2^j v$ with $u$ running modulo $2^j$ and $v$ running modulo $2^{k-j}$.  Then
\begin{equation}
 G\Big(\frac{a}{2^k}\Big) = 
 \sum_{u \shortmod{2^j}} e_{2^k}(au^2) 
 \sum_{v \shortmod{2^{k-j}}} 
 e_{2^{k-j-1}}(a uv).
\end{equation}
The inner sum over $v$ vanishes unless $u \equiv 0 \pmod{2^{k-j-1}}$, so we change variables $u = 2^{k-j-1} r$, with $r$ now running modulo $2^{2j-k+1}$.  This gives 
\begin{equation}
 G\Big(\frac{a}{2^k}\Big) = 
  2^{k-j} \sum_{r \shortmod{2^{2j-k+1}}} e_{2^{2j-k+2}}(ar^2).
\end{equation}
In the case that $k$ is even, we make the choice $j=k/2$, giving
\begin{equation}
 G\Big(\frac{a}{2^k}\Big) = 
  2^{k/2} \sum_{r \shortmod{2}} e_4(ar^2) = 2^{k/2}(1 + e_4(a)).
\end{equation}
If $k$ is odd, we take $j = \frac{k+1}{2}$, giving now
\begin{equation}
G\Big(\frac{a}{2^k}\Big) = 
 2^{\frac{k-1}{2}} \sum_{r \shortmod{2^2}} e_{2^3}(ar^2) = 2^{\frac{k+1}{2}} e_{8}(a).
\end{equation}
%Note that $|G(a/2^k)| = 2^{\frac{k+1}{2}}$ for $a$ odd, and both cases of $k$ even or odd, since $|1+e_4(a)| = \sqrt{2}$.
Assembling the above facts, and using \eqref{eq:GaussSumQuadraticOddModulusEvaluation}, now completes the proof.
%Now we can evaluate $G(a/c)$ in all cases.  Say $c= 2^k c_o$ with $k \geq 2$.  Then
%\begin{equation}
% G(a/c) = G(2^k a/c_o) G(ac_o/2^k) = \epsilon_{c_o} c^{1/2} \Big(\frac{2^k a}{c_o}\Big) 
% \begin{cases}
%   1 + e_4(ac_o), \qquad k \text{ even} \\
%   2^{1/2} e_8(ac_o), \qquad k \text{ odd}.
% \end{cases}
%\end{equation}
%We can write this as
%\begin{equation}
% G\Big(\frac{a}{c}\Big) = c^{1/2} \Big(\frac{a}{c_o}\Big) g(a,c_o),
%\end{equation}
%where for $k$ even $g$ is a function on $(\mz/4\mz)^{*} \times (\mz/4\mz)^{*}$ and if $k$ is odd $g$ is a function on $(\mz/8\mz)^{*} \times (\mz/8\mz)^{*}$.
\end{proof}

\begin{mylemma}
\label{lemma:CharacterSumArithmeticProgressionEvaluation}
 Let $\chi$ be a Dirichlet character modulo $q$, and suppose $d|q$ and $(a,d) = 1$.  Let
 \begin{equation}
  S_{\chi}(a,d,q) = \sum_{\substack{n \shortmod{q} \\ n \equiv a \shortmod{d}}} \chi(n).
 \end{equation}
Suppose that $\chi$ is induced by the primitive character $\chi^*$ modulo $q^*$, and write $\chi = \chi^* \chi_0$ where $\chi_0$ is trivial modulo $q_0$, with $(q_0, q^*) = 1$.  
Then $S_{\chi}(a,d,q) = 0$ unless $q^*|d$ in which case 
\begin{equation}
\label{char-sum-result} S_{\chi}(a,d,q) = \frac{q}{d} \chi^*(a) \prod_{\substack{p | q_0 \\ p \nmid d}} \Big(1-\frac{1}{p}\Big).
\end{equation}
\end{mylemma}
\begin{proof}
 Suppose $q = q_1 q_2$ with $(q_1, q_2) = 1$ and correspondingly factor $d = d_1 d_2$ and $\chi = \chi_1 \chi_2$ with $\chi_i$ modulo $q_i$.  The Chinese remainder theorem gives $S_{\chi}(a,d,q) = S_{\chi_1}(a, d_1, q_1) S_{\chi_2}(a, d_2, q_2)$.  Writing $d= d^* d_0$ where $d^* | q^*$ and $d_0 | q_0$, we apply this with
  $q_1 = q^*$, $q_2 = q_0$, $\chi_1 = \chi^*$, $\chi_2 = \chi_0$, $d_1 = d^*$, and $d_2 = d_0$.
By the multiplicativity of the right hand side of \eqref{char-sum-result}, it suffices to prove it for $\chi^*$ and $\chi_0$.

By \cite[(3.9)]{IK}, $S_{\chi^*}(a, d^*, q^*) = 0$ unless $q^*|d^*$, in which case it is given by \eqref{char-sum-result}, so this case is done.
 
For the $\chi_0$ part, we simply use M\"obius inversion, giving
\begin{equation}
S_{\chi_0}(a, d_0, q_0) = 
\sum_{\ell | q_0} \mu(\ell)
\sum_{\substack{n \shortmod{q_0/\ell} \\ \ell n \equiv a \shortmod{d_0}}} 1.
\end{equation}
Since $(a,d_0) = 1$ by assumption, this means that we may assume $(\ell, d_0) = 1$, and then $n$ is uniquely determined modulo $d_0$, which divides $q_0/\ell$, giving
\begin{equation}
S_{\chi_0}(a, d_0, q_0) = 
\frac{q_0}{d_0} \sum_{\substack{\ell | q_0 \\ (\ell, d_0) =1}}  \frac{\mu(\ell)}{\ell} 
= \frac{q_0}{d_0}
\prod_{\substack{p | q_0 \\ p \nmid d_0}} \Big(1-\frac{1}{p}\Big). \qedhere
\end{equation}
\end{proof}

For $a,b,c \in \mz$ with $c \geq 1$, define
\begin{equation}
 T(a,b;c) = \sum_{x,y \shortmod{c}} S(x^2, y^2;c) e_c(2xy + ax + b y).
\end{equation}
For $c_o$ odd, write its prime factorization as $c_o = \prod_{p} p^{a_p} \prod_q q^{b_q}$ where each $a_p$ is odd and each $b_q$ is even.  Let $c^* = \prod_p p$ and $c_{\square} = \prod_q q$.  Then $c^*$ is the conductor of the 
Jacobi symbol  $(\tfrac{\cdot}{c_o})$.

\begin{mylemma}
\label{lemma:TabcCalculation}
Let $a, b, c \in \mz$, with $c \geq 1$.  Suppose $c= 2^j c_o$ with $j \geq 4$ and $c_o$ odd, with $\delta$ defined as in \eqref{delta-parity}.
Define $a' = \frac{a}{(a,c)}$, $b' = \frac{b}{(b,c)}$.
Then
$T(a,b;c) = 0$ unless $4|(a,b)$ and $(a,c) = (b,c)$, in which case
\begin{equation}
\label{eq:Tabcformula}
T(a,b;c)  = (a,\frac{c}{2^{2+\delta}}) c^{3/2} e_c(-ab/4)   \Big(\frac{a'b'}{c^*}\Big) g_{\delta}(a',b', c_o)
\delta(c^*|\frac{c_o}{(a,c_o)})
\prod_{\substack{p | c_{\square}, \thinspace p \nmid \frac{c_o}{(a,c_o)}}} (1-p^{-1}),
\end{equation}
where $g_{\delta}$ is some function depending on $a',b',c_o$ modulo $2^{2+\delta}$ that additionally depends on $(\frac{2^j}{(a,2^j)}, 2^{2+\delta})$. In particular, we have that $T(0,b;c) \ll c^{5/2} \delta(c^* = 1) \delta(c|b)$.
\end{mylemma}
\begin{proof}
We have
\begin{equation}
 T(a,b;c) = 
 \sumstar_{t \shortmod{c}}
 \sum_{x,y \shortmod{c}}
  e_c(t(x + \overline{t} y)^2 + ax + by).
\end{equation}
Changing variables $x \rightarrow x- \overline{t} y$ and evaluating the resulting $y$-sum by orthogonality, we deduce
\begin{equation}
 T(a,b;c) = c
 \sumstar_{\substack{t \shortmod{c} \\ bt \equiv a \shortmod{c}}}
 \sum_{x \shortmod{c}}
  e_c(tx^2 + ax).
\end{equation}
The congruence in the sum implies that $T(a,b;c) = 0$ unless $(a,c) = (b,c)$, a condition that we henceforth assume.  Changing variables $x \rightarrow x + c/2$ also shows that $T(a,b;c) = 0$ unless $2|a$, so we assume this condition also.  

Write $c$ uniquely as $c = c_1 c_2$ where $c|c_1^2$, $c_2 | c_1$ and $c_1/c_2$ is square-free
(another way to see this factorization is by writing $c$ uniquely as $AB^2$ with $A$ square-free; then $c_1 = AB$ and $c_2 = B$).
Observe that $2^2|c_2$ from $2^4|c$.  Let $x = x_1 + c_1 x_2$, and let $Q(x) = tx^2 + ax$.  Note that 
\begin{equation}
Q(x_1 + c_1 x_2) = Q(x_1) + Q'(x_1) c_1 x_2 + \tfrac{Q''(x_1)}{2} c_1^2 x_2^2 \equiv Q(x_1) + Q'(x_1) c_1 x_2 \pmod{c}.
\end{equation}
Thus
\begin{equation}
 \sum_{x \shortmod{c}} e_c(Q(x)) = 
 \sum_{x_1 \shortmod{c_1}} e_c(Q(x_1))
 \sum_{x_2 \shortmod{c_2}} e_{c_2}(Q'(x_1) x_2)
 =
 c_2 \sum_{\substack{x_1 \shortmod{c_1} \\ Q'(x_1) \equiv 0 \shortmod{c_2}}} e_c(Q(x_1)).
\end{equation}
In our case, $Q'(x_1) = 2tx_1 + a$, so the congruence means $2x_1 \equiv - \overline{t} a \pmod{c_2}$.  Since $2|a$ and $2|c_2$, this is equivalent to $x_1 \equiv - \overline{t} \frac{a}{2} \pmod{c_2/2}$.  Writing $x_1 = - \overline{t} \frac{a}{2} + \frac{c_2}{2} v$, with $v$ running modulo $2 \frac{c_1}{c_2}$, we obtain
\begin{equation}
 \sum_{x \shortmod{c}} e_c(Q(x))
 = c_2 e_c(- \overline{t} a^2/4)
 \sum_{v \shortmod{2 \frac{c_1}{c_2}}} 
 e\Big(\frac{t v^2}{4 c_1/c_2}\Big).
\end{equation}
While the exponential in the inner sum has modulus $4c_1/c_2$, the sum is only over $0 \le v \le 2(c_1/c_2) -1$.  However, observe that the exponential has the same values at $1\le -v \le 2(c_1/c_2)$, so that the inner sum above is half of a Gauss sum. Thus 
\begin{equation}
 T(a,b;c) = c \frac{c_2}{2} 
 \sumstar_{\substack{t \shortmod{c} \\ bt \equiv a \shortmod{c}}}
e_c(- \overline{t} a^2/4)
 G\Big(\frac{t}{4 c_1/c_2}\Big).
 \end{equation}
By Lemma \ref{lemma:Gcalc}, we deduce
\begin{equation}
\label{eq:TabcFormula}
 T(a,b;c) = c^{3/2} \epsilon_{c_o}
 \sumstar_{\substack{t \shortmod{c} \\ bt \equiv a \shortmod{c}}}
e_c(- \overline{t} a^2/4)
 \Big(\frac{t 2^{\delta}}{c_{o}}\Big)
\begin{cases}
 1 + e_4(t c_o), \qquad \delta=0 \\
 2^{1/2} e_8(t c_o), \qquad \delta=1.
\end{cases}
\end{equation}
This formulation contains a few additional observations. We have used that the Jacobi symbol $(\frac{t}{(c_1/c_2)_o})$ agrees with $(\frac{t}{c_o})$ for $t$ coprime to $c$, where $n_o$ is the odd part of an integer $n$. We have also used that $(c_1/c_2)_o$ and $c_o$ have the same values modulo $8$. Thus we can replace $\epsilon_{(c_1/c_2)_o}, e_4(t (c_1/c_2)_o)$, and $e_8(t (c_1/c_2)_o)$ with $\epsilon_{c_o}, e_4(t c_o)$, and $e_8(t c_o)$ respectively. These observations can easily be checked by using multiplicativity to reduce to the case when $c$ is a power of an odd prime. If $c=p^l$, then $c_1/c_2=1$ when $l$ is even, and $c_1/c_2=p$ when $l$ is odd.

Next we turn to the $t$-sum in \eqref{eq:TabcFormula}.  Suppose first that $2||a$. Let $a' = \frac{a}{(a,c)}$, $b' = \frac{b}{(a,c)}$. The congruence $bt \equiv a \pmod{c}$ uniquely determines $t$ modulo $c/(a,c)$, since it is equivalent to $\overline{t} \equiv b'\overline{a'} \pmod{c/(a,c)}$. Now in the $t$-sum, one can pair up $\overline{t}$ with $\overline{t}+c/2$ and observe that the corresponding values of the exponential $e_c(- \overline{t} a^2/4)$ will cancel out since $e_c(- (c/2) a^2/4)=-1$. Also, the values of $(\frac{t}{c_o})=(\frac{\overline{t}}{c_o})$, $e_4(t c_o)=e_4(\overline{t} c_o)$, and $e_8(t c_o)=e_8(\overline{t} c_o)$ remain the same under $\overline{t}\to \overline{t}+c/2$, since by assumption $2^4|c$. Therefore, $T(a,b,c)$ vanishes unless $4|a$ (and hence $4|b$), which we now assume to be the case. This allows the convenient simplification $e_c(-\overline{t} a^2/4) = e_c(-ab/4)$. 

 Breaking up the $t$-sum into congruence classes modulo $2^{2+\delta}$, to uniquely determine $e_{2^{2+\delta}}(t c_o)$, we obtain
\begin{equation}
 T(a,b;c) = c^{3/2} \epsilon_{c_o}
 e_c(- ab/4)
 \sumstar_{v \shortmod{2^{2+\delta}}}
 \begin{bothcases}
 1 + e_4(v c_o) \\
 2^{1/2} e_8(v c_o)
\end{bothcases} 
 \sumstar_{\substack{t \shortmod{c} \\ t \equiv \overline{b'} a' \shortmod{\frac{c}{(a,c)}}
 \\ t \equiv v \shortmod{2^{2+\delta}}
 }}
 \Big(\frac{t 2^{\delta}}{c_{o}}\Big).
\end{equation}
For the congruence $t \equiv \overline{b'}a' \pmod{\frac{c}{(a,c)}}$  to be consistent with $t \equiv v \pmod{2^{2+\delta}}$, it is necessary and sufficient that $v \equiv \overline{b'}a' \pmod{(\frac{c}{(a,c)}, 2^{2+\delta})}$.

Recall that $c = 2^j c_o$, where $j \geq 4$. 
Factoring the moduli in the sum, we have
\begin{equation}
 \sumstar_{\substack{t \shortmod{c} \\  t \equiv \overline{b'} a' \shortmod{\frac{c}{(a,c)}}
 \\ t \equiv v \shortmod{2^{2+\delta}}
 }}
 \Big(\frac{t}{c_{o}}\Big)
 =
\Big( \sumstar_{\substack{t \shortmod{c_o} \\ t \equiv \overline{b'} a' \shortmod{\frac{c_o}{(a,c_o)}}}}
 \Big(\frac{t}{c_{o}}\Big)
 \Big)
 \Big(\sumstar_{\substack{t \shortmod{2^j} \\ t \equiv \overline{b'} a' \shortmod{\frac{2^j}{(a,2^j)}}
 \\ t \equiv v \shortmod{2^{2+\delta}}}} 1 \Big).
\end{equation}
The sum modulo $2^j$ above equals, by the Chinese Remainder Theorem and the fact that the condition $(t,2)=1$ is automatic because $(v,2)=1$,
$$
\frac{2^j}{[\frac{2^j}{(a,2^j)}, 2^{2+\delta}]} 
= 
\frac{2^j (\frac{2^j}{(a,2^j)}, 2^{2+\delta})}{\frac{2^j}{(a,2^j)} 2^{2+\delta}}
= 
(a,2^{j-2-\delta}),
$$
provided of course that $v \equiv \overline{b'} a' \pmod{(\frac{2^j}{(a,2^j)}, 2^{2+\delta})}$. 
Therefore, we have that $T(a,b;c)$ equals
\begin{equation}
\label{eq:TsumEvaluationTowardsTheEnd}
c^{3/2} \epsilon_{c_o}
 e_c(- ab/4) (a, 2^{j-2-\delta})
\sumstar_{\substack{v \shortmod{2^{2+\delta}} \\ v \equiv \overline{b'} a' \shortmod{(\frac{2^j}{(a,2^j)}, 2^{2+\delta})}}}
 \begin{bothcases}
 1 + e_4(v c_o) \\
 2^{1/2} e_8(v c_o)
\end{bothcases}
 \sumstar_{\substack{t \shortmod{c_o} \\ t \equiv \overline{b'}a' \shortmod{\frac{c_o}{(a,c_o)}}
 }}
 \Big(\frac{t 2^{\delta}}{c_{o}}\Big).
\end{equation}
By Lemma \ref{lemma:CharacterSumArithmeticProgressionEvaluation} with $q= c_o$, $d = \frac{c_o}{(a,c_o)}$, $a = \overline{b'} a'$, $q^* = c^*$, and $q_0 = c_{\square}$, we have 
\begin{equation}
\label{eq:tsumevaluation}
\sumstar_{\substack{t \shortmod{c_o} \\ t \equiv \overline{b'}a' \shortmod{\frac{c_o}{(a,c_o)}}
 }}
 \Big(\frac{t}{c_{o}}\Big)
=  (a,c_o) \Big(\frac{a' b'}{c^*}\Big) 
\delta(c^*|\frac{c_o}{(a,c_o)})
\prod_{\substack{p | c_{\square} \\ p \nmid \frac{c_o}{(a,c_o)}}} (1-p^{-1})
 .
\end{equation}
Inserting \eqref{eq:tsumevaluation} into \eqref{eq:TsumEvaluationTowardsTheEnd} and simplifying a bit using $(a, c_o)(a, 2^{j-2-\delta}) = (a, \frac{c}{2^{2+\delta}})$, we deduce that
 $T(a,b;c)$ equals
\begin{equation}
 c^{3/2} \epsilon_{c_o}
 e_c(- ab/4) (a,\tfrac{c}{2^{2+\delta}})
 \Big(\frac{a' b'  2^{\delta}}{c^*}\Big)
 \prod_{\substack{p|c_{\square} \\ p \nmid \frac{c_o}{(a,c_o)}}} (1-p^{-1})
 \sumstar_{\substack{v \shortmod{2^{2+\delta}} \\ v \equiv \overline{b'} a' \shortmod{(\frac{2^j}{(a,2^j)}, 2^{2+\delta})}}}
 \begin{cases}
 1 + e_4(v c_o) \\
 2^{1/2} e_8(v c_o),
\end{cases}
\end{equation}
times the delta function that $c^*$ divides $\frac{c_o}{(a,c_o)}$.
The inner sum over $v$ is a function of $a', b', c_o$ modulo $2^{2+\delta}$ that additionally depends 
on $(\frac{2^j}{(a,2^j)}, 2^{2+\delta})$. In addition, $(\frac{2^{\delta}}{c^*})$ is a function of $c_o$ modulo $2^{2+\delta}$.
%We can slightly simplify this, for fun.  
% First consider the case $\delta = 1$.  The $v$-sum vanishes unless the congruence uniquely determines $v$ modulo $8$, which 
% occurs precisely when $2^{3} | \frac{2^j}{(a,2^j)}$.  This is equivalent to $v_2(a) \leq j-3$.  
% Somewhat similarly, if $\delta=0$ then the $e_4(v c_o)$ part sums to zero unless $v$ is uniquely determined modulo $4$, which means that $v_2(a) \leq j-2$. 
\end{proof}

\section{Start of proof}\label{section:startofproof}
Let $0\le U \le (2-\delta)T$. By an approximate functional equation, dyadic decomposition of unity, and Cauchy's inequality, we have
\begin{equation}
\label{eq:momentFirstFormula}
\mathcal{M}:= \sum_{T < t_j <T + \Delta} 
 |L(\mathrm{sym}^2 u_j, 1/2+iU)|^2
 \ll
 \max_{1 \ll N \ll N_{\text{max}}} \frac{ T^{\varepsilon}}{N}
 \sum_{T < t_j <T + \Delta} 
 \Big|
 \sum_{n} \frac{\lambda_j(n^2)}{n^{iU}} w_N(n) \Big|^2,
\end{equation}
where $w_N(x)$ is supported on $x \asymp N$ and satisfies $w_N^{(j)}(x) \ll_j  N^{-j}$ 
and $N_{\text{max}} = (U+1)^{1/2} T^{1+\varepsilon}$.  
To save some clutter in the notation, we want to simply write $U$ instead of $U+1$ in all estimates involving $U$. The reader may accept this as a convention or, when $0\le U\le 1$, we can write $n^{-iU} w_N(n) = n^{-i(U+1)} n^{i} w_N(n)$ and absorb $n^i$ into $w_N(n)$ by redefining the weight function. Thus we can henceforth assume that $U\ge 1$.

Next we insert a weight
\begin{equation}
 \label{eq:h-definition} h(t) =
\frac{t^2 + \frac14}{T^2}
\Big[  
\exp\Big(-\frac{(t-T)^2}{\Delta^2} \Big) + \exp\Big(-\frac{(t+T)^2}{\Delta^2} \Big)
\Big],
\end{equation}
write $\lambda_j(n^2)=\rho_j(n^2)/\rho_j(1)$ and over-extend (by positivity) the spectral sum to  an orthonormal basis of all cusp forms of level $2^4$, embedding the level $1$ forms. This  embedding trick, introduced for the purpose of simplifying the $2$-part of the exponential sum in Lemma \ref{lemma:TabcCalculation}, is motivated from \cite[p.4]{Bl2}. We also form the obvious Eisenstein series variant on the sum.  This leads to 
the inequality (see the remarks following Lemma \ref{lemma:Kuztrace})
\begin{multline}
 \mathcal{M} \ll  \max_{1 \ll N \ll N_{\text{max}}}
 \frac{T^{\varepsilon}}{N}
 \Big(
 \sum_{u_j \text{ level $2^4$}}  \frac{h(t_j)}{\cosh(\pi t_j)}
 \Big|
 \sum_{n} \frac{\rho_j(n^2)}{n^{iU}} w_N(n) \Big|^2
\\
 + \sum_{\mathfrak{a}}  \frac{1}{4 \pi} \intR  \frac{h(t)}{\cosh(\pi t)}
 \Big|
 \sum_{n} \frac{\tau_{\mathfrak{a}, it}(n^2)}{n^{iU}} w_N(n) \Big|^2
 dt
 \Big).
\end{multline}
Opening the square and applying the Kuznetsov formula, we obtain
\begin{equation}\label{after-kuznetsov}
 \mathcal{M} \ll \Delta T^{1+\varepsilon} +  \max_{1 \ll N \ll N_{\text{max}}} T^{\varepsilon} |\mathcal{S}(H)|,
\end{equation}
where 
\begin{equation}
 \mathcal{S}(H) = 
 \frac{1}{N} \sum_{c \equiv 0 \shortmod{2^4}}
 \sum_{m,n} \frac{S(m^2, n^2;c)}{c m^{iU} n^{-iU}} 
 w_N(m) w_N(n)
 H\Big(\frac{4 \pi mn}{c}\Big),
\end{equation}
\begin{equation}
\label{Hdef} H(x) = i \intR J(x,t) t \tanh(\pi t) h(t) dt,
\end{equation}  and $J(x,t)$ is as defined in Lemma \ref{lemma:Kuztrace}.

By \cite[(3.10)]{JM} we get that $H(x) \ll \frac{\Delta}{T} x^{2}$ for $x\le 1$. Using this with $x=4\pi mn/c$, we can truncate $c$ at some large power of $T$, say $c \le T^{100}$, with an acceptable error term.

Using \cite[8.411 11]{GR} and the fact that the integrand in \eqref{Hdef} is an even function of $t$, one can derive as in \cite[(3.13)]{JM} that $H(x) = \frac{2}{\pi } \mathrm{Re}(H_{0}(x))$, where
\begin{equation}
 H_0(x)  = \intR e^{ix \cosh{v}}\intR e^{-2ivt} t \tanh(\pi t) h(t) dt dv.
\end{equation}
The inner $t$-integral above is 
\begin{multline}
 \intR e^{-2ivt} t \tanh(\pi t) \frac{t^2 + \frac14}{T^2} \Big(\exp\Big(-\frac{(t-T)^2}{\Delta^2}\Big)+\exp\Big(-\frac{(t+T)^2}{\Delta^2} \Big)\Big) dt
\\
 = \Delta T( e^{-2ivT}+e^{2ivT}) g(\Delta v),
\end{multline}
where $g(y) = g_{\Delta, T}(y)$ behaves like a fixed (even) Schwartz-class function; namely it satisfies the derivative bounds $g^{(j)}(y) \ll_{j,A} (1+|y|)^{-A}$, for any $j, A \in \mathbb{Z}_{\geq 0}$.  Hence
\begin{equation}
\label{eq:H0formula}
 H_{0}(x)   = 2 \Delta T \intR e^{ix \cosh{v}} e^{-2ivT} g(\Delta v) dv.
\end{equation}
From this, we can write the real part of $H_0(x)$ as a linear combination of $H_{\pm}(x)$, where
\begin{equation}
 H_{\pm}(x) = \Delta T \intR e^{\pm i x \cosh{v} - 2ivT} g(\Delta v) dv
  = \Delta T e^{\pm ix} \intR e^{\pm ix (\cosh{v} - 1) - 2ivT} g(\Delta v) dv.
\end{equation}
 Then \eqref{after-kuznetsov} becomes
\begin{equation}\label{after-kuznetsov-2}
 \mathcal{M} \ll \Delta T^{1+\varepsilon} +  \max_{\substack{1 \ll N \ll N_{\text{max}}\\ \pm }} T^{\varepsilon} |\mathcal{S}(H_\pm)|.
\end{equation}
It suffices to bound $\mathcal{S}(H_{+})$, as the argument for $\mathcal{S}(H_{-})$ is similar.
For convenience, let us write this as $H_{+}(x) = \Delta T e^{ix} K_{+}(x)$, where
\begin{equation}
\label{eq:K+formula}
 K_{+}(x) =   \intR e^{ix (\cosh{v} - 1) - 2ivT} g(\Delta v) dv.
\end{equation}

Finally, we apply a dyadic partition of unity to the $c$-sum.
To summarize, we have shown
\begin{equation}
\label{eq:SH+formula}
 \mathcal{S}(H_{+}) = \frac{\Delta T}{N} 
 \sum_{C } 
 \sum_{c \equiv 0 \shortmod{2^4}}
 \sum_{m,n} \frac{S(m^2, n^2;c) e_c(2mn)}{c m^{iU} n^{-iU}} 
 w(m,n,c)
 K_{+}\Big(\frac{4 \pi mn}{c}\Big) + O(T^{-100}),
\end{equation}
where the first sum is a sum over integers $C$ equal to $2^{j/2}$ for $0\le j\leq 300 \log T$ and $w(x_1,x_2,x_3)=w_{N,C}(x_1,x_2,x_3)$ is 1-inert and supported on  $x_1\asymp x_2\asymp N$ and $c\asymp C$.

We may approximate $H_{+}(x)$ quite well by truncating the integral at $|v| \leq \Delta^{-1} T^{\varepsilon}$, and then use an integration by parts argument to see that $H_{+}(x)$ is very small unless
\begin{equation}
 x \gg \Delta T^{1-\varepsilon}.
\end{equation}
For more details of an alternative approach, one may see \cite[pp.76-77]{JM}.
In our situation where $x \asymp \frac{N^2}{C}$, we conclude that we may assume
\begin{equation}
\label{eq:cUpperBound}
C \ll T^{\varepsilon} \frac{N^2}{\Delta T^{}} \ll T^{\varepsilon} \frac{UT}{\Delta}.
\end{equation}
For our purposes it is inconvenient to develop the $v$-integral further at this early stage.
However, we do record the following slight refinement that is useful for large values of $x$.
\begin{mylemma}
\label{lemma:K+cutoff}
Suppose that
\begin{equation}
x \gg T^{2-\varepsilon}.
\end{equation}
Then
\begin{equation}
\label{eq:K+cutoff}
K_{+}(x) = \int_{|v| \ll  x^{-1/2} T^{\varepsilon}} 
e^{ix(\cosh(v) - 1) - 2iTv} g(\Delta v) \eta(v) dv + O((xT)^{-100}),
\end{equation}
where $\eta$ is supported on $|v| \ll x^{-1/2}T^{\varepsilon}$ and satisfies property \eqref{inertproperty} for a $1$-inert function.
\end{mylemma}
\begin{proof}
This follows from the integration by parts lemma.
\end{proof}

\section{Double Poisson summation}
Next we apply Poisson summation to the $m$ and $n$ sums in \eqref{eq:SH+formula}, giving
\begin{equation}
  \mathcal{S}(H_{+}) = \frac{\Delta T}{N}  \sum_C
\sum_{c \equiv 0 \shortmod{2^4}}
  \sum_{k,\ell} \frac{T(-k,\ell;c)}{c^3} 
I(k,\ell,c) + O(T^{-100}),
\end{equation}
where
%\begin{equation}
% a(k,\ell;c) = \frac{1}{c^2} \sum_{x, y \shortmod{c}} S(x^2,y^2;c) e_c(2xy) e_c(-kx+\ell y),
%\end{equation}
%and
\begin{equation}
\label{eq:Ikellcdef}
 I(k,\ell,c) = \int_0^{\infty} \int_0^{\infty} 
 x^{-iU} y^{iU} e_c(kx - \ell y) K_{+}\Big(\frac{4 \pi xy}{c}\Big)
 w(x,y,c) dx dy.
\end{equation}

By Lemma \ref{lemma:TabcCalculation}, $T(-k,\ell;c) = 0$ unless $(k,c) = (\ell, c)$ and $4|(k,\ell)$, in which case
\begin{equation}
 T(-k,\ell,c) = c^{3/2} (k, 2^{-2-\delta} c) 
 e_c(k\ell/4)   \Big(\frac{k' \ell'}{c^*}\Big) g_{\delta}(k',\ell', c_o) \delta(c^*|\frac{c_o}{(k,c_o)})
 \prod_{\substack{p | c_{\square}, \thinspace p \nmid \frac{c_o}{(k,c_o)}}} (1-p^{-1})
 ,
\end{equation}
where $k' = \frac{k}{(k,c)}$, $\ell' = \frac{\ell}{(\ell,c)}$, $\delta$ was defined in \eqref{delta-parity}, and other notation is carried over from Lemma \ref{lemma:TabcCalculation} (here the function $g_{\delta}$ has the same properties as the one appearing in Lemma \ref{lemma:TabcCalculation}, but may not agree with it).

Write 
\begin{equation}
c = 2^{\lambda} c_o, \qquad k = 2^{\nu} k_o \qquad \ell = 2^{\gamma} \ell_o, 
\end{equation}
with $(k_o \ell_o c_o, 2) = 1$.  The condition $(k,c) = (\ell, c)$ now becomes $\min(\lambda, \nu) = \min(\lambda, \gamma)$, and $(k_o, c_o) = (\ell_o, c_o)$.  The condition $4|(k,\ell)$ now means $\nu, \gamma \geq 2$.  
We also write 
\begin{equation}
c_o = q r_1^2 r_2^2
\end{equation}
where $q$ is square-free, $r_1 | q^{\infty}$, and $(q,r_2) = 1$.  With this notation, $c^* = q$ and $c_{\square}$ shares the same prime factors as $r_2$.  
Note $\frac{c_o}{(k, c_o)} = \frac{q r_1^2}{(k_o, q r_1^2)} \frac{r_2^2}{(k_o, r_2^2)}$.  Thus the condition $ c^* | \frac{c_o}{(k,c_o)}$ means $q | \frac{q r_1^2}{(k_o, q r_1^2)}$, which is equivalent to $(k_o, q r_1^2) | r_1^2$.
Then
\begin{multline}
\label{eq:SH+formula2}
  \mathcal{S}(H_{+}) = \sum_C \frac{\Delta T}{NC^{3/2}} 
\sum_{\substack{\nu, \gamma \geq 2, \thinspace \lambda \geq 4 \\ \min(\lambda, \nu) = \min(\lambda, \gamma)}} 
(2^{\nu}, 2^{\lambda-2-\delta})
\sum_{\substack{ (r_1 r_2, 2) = 1}}
\sumstar_{\substack{q: r_1 | q^{\infty} \\ (q,2r_2) = 1}}
\sum_{\substack{(k_o \ell_o, 2)=1 \\ (k_o, c_o) = (\ell_o, c_o) \\ (k_o, q r_1^2) | r_1^2}}
\\
\Big(\prod_{\substack{p | r_2, \thinspace p \nmid \frac{r_2^2}{(k_o,r_2^2)}}} (1-p^{-1}) \Big)
\Big(\frac{k' \ell'}{c^*}\Big)
  (k_o, c_o)
  e_c(k \ell/4) 
  g_{\delta}(k',\ell', c_o)
I(k,\ell,c) + O(T^{-100}),
\end{multline}
where in places to simplify the notation we did not display the substituted values such as $c_o = q r_1^2 r_2^2$.   We remark that the statement that $g_{\delta}(k', \ell', c_o)$ depends additionally on $(\frac{c}{(a,c)}, 2^{2+\delta})$ means it depends on $(2^{\lambda-\min(\lambda,\nu)}, 2^{2+\delta})$.  In particular, $g_{\delta}$ depends additionally on $\lambda, \nu$, but only lightly, in the sense that it falls in the four following cases:
\begin{equation}
\label{eq:alphabetalines}
\text{i) }  \lambda \leq \nu, \qquad 
\text{ii) } \lambda = \nu + 1, \qquad  
\text{iii) } \lambda = \nu + 2, \qquad
\text{iv) } \lambda \geq \nu + 3.
\end{equation}

Next we want to give a variable name to $(k_o, c_o)$, etc.  We have $(k_o, c_o) = (k_o, q r_1^2) (k_o, r_2^2)$, and similarly $(\ell_o, c_o) = (\ell_o, qr_1^2)(k_o, r_2^2)$.  Let 
\begin{equation}
(k_o, q r_1^2) = (\ell_o, q r_1^2) = g_1, 
\qquad \text{and} \qquad
(k_o, r_2^2) = (\ell_o, r_2^2) = g_2.
\end{equation}
Here $g_1$ runs over divisors of $r_1^2$ and $g_2$ runs over divisors of $r_2^2$.  Let
\begin{equation}
k_o = g_1 g_2 k_o', \qquad \ell_o = g_1 g_2 \ell_o',
\end{equation}
where
$(k_o' \ell_o', q \frac{r_1^2}{g_1}) = 1$ and $(k_o' \ell_o', \frac{r_2^2}{g_2}) = 1$.
In our context, the presence of the Jacobi symbol $(\frac{k' \ell'}{q})$ means that we may automatically assume $(k_o' \ell_o', q) = 1$ which implies $(k_o' \ell_o', q \frac{r_1^2}{g_1}) = 1$.
Note that $k' = k_o' 2^{\nu- \min(\nu, \lambda)}$ and $\ell' = \ell_o' 2^{\gamma - \min(\gamma, \lambda)}$.  We also apply quadratic reciprocity, giving $(\frac{k_o' \ell_o'}{q}) = (\frac{q}{k_o' \ell_o'})$ times a function depending on $k_o', \ell_o', q'$ modulo $8$ (which only alters the definition of $g$).
Making these substitutions, we obtain
\begin{multline}
\label{eq:SH+formula3}
  \mathcal{S}(H_{+}) = \sum_C \frac{\Delta T}{NC^{3/2}} 
\sum_{\substack{\nu, \gamma \geq 2, \thinspace \lambda \geq 4 \\ \min(\lambda, \nu) = \min(\lambda, \gamma)}} 
(2^{\nu}, 2^{\lambda-2-\delta})
\sum_{\substack{ (r_1 r_2, 2) = 1}}
\sum_{\substack{g_1 |r_1^2 \\ g_2 | r_2^2}} g_1 g_2
\prod_{\substack{p | r_2, \thinspace p \nmid \frac{r_2^2}{g_2}}} (1-p^{-1})
\\
\sumstar_{\substack{q: r_1 | q^{\infty} \\ (q,2r_2) = 1}}
\sum_{\substack{(k_o' \ell_o', 2)=1 \\ (k_o' \ell_o', \frac{r_2^2}{g_2}) = 1}}
\Big(\frac{q}{k_o' \ell_o'}\Big)
  e_c(k \ell/4)
  g_{\lambda, \nu, \gamma, \delta}(k_o',\ell_o', q)
I(k,\ell,c) + O(T^{-100}),
\end{multline}
where $g_{\lambda,\nu,\gamma,\delta}$ is some new function modulo $8$. 

Finally, we decompose $g$ into Dirichlet characters modulo $8$, and break up the sum according to the four cases in \eqref{eq:alphabetalines}, leading to a formula of the form
\begin{equation}
|\mathcal{S}(H_{+})|  \ll 
\max_{\substack{\eta_1, \eta_2, \eta_3 \\ \text{cases in } \eqref{eq:alphabetalines}}}  |\mathcal{S}_{\eta}(H_{+})|,
\end{equation}
where 
\begin{multline}
\label{eq:SH+formula4}
  \mathcal{S}_{\eta}(H_{+}) = \sum_C \frac{\Delta T}{NC^{3/2}} 
\sum_{\substack{\nu, \gamma \geq 2, \thinspace \lambda \geq 4 \\ \min(\lambda, \nu) = \min(\lambda, \gamma) \\ \text{one of  \eqref{eq:alphabetalines} holds}}} 
(2^{\nu}, 2^{\lambda-2-\delta})
\sum_{\substack{ (r_1 r_2, 2) = 1}}
\sum_{\substack{g_1 |r_1^2 \\ g_2 | r_2^2}} g_1 g_2
\prod_{\substack{p | r_2, \thinspace p \nmid \frac{r_2^2}{g_2}}} (1-p^{-1})
\\
\sumstar_{\substack{q: r_1 | q^{\infty} \\ (q,2r_2) = 1}}
\sum_{\substack{(k_o' \ell_o', 2)=1 \\ (k_o' \ell_o', \frac{r_2^2}{g_2}) = 1}}
\eta_1(k_o') \eta_2(\ell_o') \eta_3(q)
\Big(\frac{q}{k_o' \ell_o'}\Big)
  e_c(k \ell/4)
I(k,\ell,c) + O(T^{-100}).
\end{multline}

%\blue{Delete the following now?:
%We write this as
%\begin{equation}
%\label{eq:SintermsofS'}
%\mathcal{S}(H_{+}) \ll \sum_C \frac{\Delta T}{C^{3/2}}  
%\sum_{\substack{\nu, \gamma \geq 2, \thinspace \lambda \geq 4 \\ \min(\lambda, \nu) = \min(\lambda, \gamma)}} 
%(2^{\nu}, 2^{\lambda-2-\delta})
%\sum_{\substack{(r_1, r_2) =1 \\ (r_1 r_2, 2) = 1}}
%\sum_{\substack{g_1 |r_1^2 \\ g_2 | r_2^2}} g_1 g_2
%|\mathcal{S}'| + O(T^{-100}),
%\end{equation}
%where
%\begin{equation}
%\label{eq:S'def}
%\mathcal{S}' = 
%\sum_{\substack{ (k_o' \ell_o', 2\frac{r_2^2}{g_2}) = 1}}
%\thinspace
%\sumstar_{\substack{q: r_1 | q^{\infty} \\ (q,2) = 1}}
%\Big(\frac{k_o' \ell_o'}{q}\Big)
%  e_c(k \ell/4)
%  g_{\lambda, \nu, \gamma, \delta}(k_o',\ell_o', q)
%I(k,\ell,c),
%\end{equation}
%where we have not substituted for $k$ and $\ell$ in the smooth function aspect.  
%The above arrangment indicates that we do not seek any cancellation from the outer variables $\lambda, \nu, \gamma, r_1, r_2, g_1, g_2$, and the reader is encouraged to consider the case $r_1 = r_2 = g_1 = g_2 = 1$, and $\gamma = 4$, $\lambda=\nu =2$.
%}

\section{The behavior of $I(k,\ell,c)$}\label{poissonintegral}
\label{section:Ibehavior}
The purpose of this section is to develop the analytic properties of $I(k, \ell, c)$.  We begin with a few reduction steps.
Inserting \eqref{eq:K+formula} into \eqref{eq:Ikellcdef}, we have
\begin{equation}
\label{eq:Ikellc}
 I(k,\ell,c) = 
 \intR g(\Delta v) e^{-2ivT}
 \int_0^{\infty} \int_0^{\infty} 
 x^{-iU} y^{iU} e_c(kx - \ell y + 2 xy(\cosh{v} - 1))
 w(x,y,c) dx dy dv.
\end{equation}

Let $A,B>0$, $\epsilon\ge 0$ be real numbers and $N$ and $U$ as before, and consider the integral
\begin{equation}
 \label{eq:Idef} I(A, B, U, \epsilon, N) = \int_{\mr^2} e^{i \phi(x,y)} w_N(x,y, \cdot)dx dy,
\end{equation}
where $w_N$ is $1$-inert, supported on $x \asymp y \asymp N$ with $N\gg 1$, and
\begin{equation}
 \phi(x,y) = - U \log{x} + U \log{y} + Ax - By + \epsilon xy.
\end{equation}
In our case,
\begin{equation}
\label{eq:ABepsilondefs}
A = \frac{2 \pi k}{c}, \qquad B = \frac{2\pi \ell}{c}, \qquad \epsilon = \epsilon(v) = 4 \pi \frac{\cosh{v} -1}{c},
\end{equation}
and then
\begin{equation}
I(k, \ell, c) = \intR g(\Delta v) e^{-2ivT} I(A,B,U,\epsilon(v), N) dv.
\end{equation}
Note that in our study of $I(A, B, U, \epsilon, N)$, we may assume throughout that $\epsilon>0$, because $\epsilon(v)=0$ if an only if $v=0$, a set of measure 0 for the $v$-integral of $I(k, \ell, c)$.

Moreover, we may wish to assume that $w_N(x,y) = w_N(x,y, \cdot)$ depends on some unspecified finite list of additional variables that are held suppressed in the notation.  In this situation we will assume that $w_N$ is $1$-inert in terms of all the variables, not just $x$ and $y$.

\begin{mylemma}
\label{lemma:Iexpansion1}
 Suppose that $\epsilon N^2 = o(U)$, with $U \rightarrow \infty$.  
\\
1. Then $I(A,B,U,\epsilon,N) \ll_C N U^{-C}$ with $C>0$ arbitrarily large, unless
\begin{equation}
\label{eq:ABsize}
 A \asymp B \asymp \frac{U}{N}.
\end{equation}
\\
2. In the range \eqref{eq:ABsize}, we have
\begin{equation}
\label{eq:IasymptoticStationaryPhase}
 I = \frac{N^2}{U} e^{i\phi(x_0, y_0)} W(\cdot) + O(N^2 U^{-C}),
\end{equation}
where $(x_0, y_0)$ is the unique solution to $\nabla \phi(x_0,y_0) = {\bf 0}$, and $W$ is $1$-inert in terms of any suppressed variables on which $w_N$ may depend.
\\
3. Supposing \eqref{eq:ABsize} holds, $\phi(x_0, y_0)$ has the asymptotic expansion
\begin{equation}
\label{eq:phiTaylor}
 \phi(x_0, y_0) = U \log(A/B) + 
 \sum_{j=0}^{J} c_j U \Big(\frac{\epsilon U}{AB}\Big)^{1+2j}
 + 
 O\Big(U \Big(\frac{\epsilon U}{AB}\Big)^{3+2J} \Big),
\end{equation}
for some absolute constants $c_j$.
\end{mylemma}
Note that \eqref{eq:ABsize} implies $\frac{\epsilon U}{AB} \asymp \frac{\epsilon  N^2}{U} = o(1)$, so that \eqref{eq:phiTaylor} is an asymptotic expansion.  We also remark that the assumption $\epsilon N^2 = o(U)$ means that the dominant part of $\phi$ comes from $-U \log{x} + U \log{y}$, and $\epsilon xy$ is a smaller perturbation.

\begin{proof}
 The integration by parts lemma (Lemma \ref{lemma:exponentialintegral}) shows the integral is small unless \eqref{eq:ABsize} holds.  Assuming \eqref{eq:ABsize} holds, then Lemma \ref{lemma:exponentialintegralStatPhase} may be iteratively applied (using the remarks following Lemma \ref{lemma:exponentialintegralStatPhase}) which gives the form \eqref{eq:IasymptoticStationaryPhase}, with a $1$-inert function $W$.

 It only remains to derive the Taylor expansion for $\phi(x_0, y_0)$.  We have
 \begin{equation}
 \phi(Ux/A, Uy/B) = U \log(A/B) + U \Phi(x,y), 
 \end{equation}
where
\begin{equation}
\Phi(x,y) = -\log{x} + \log{y} + x-y + \delta xy, \quad \text{and} \quad \delta = \frac{\epsilon U}{AB} = o(1).
\end{equation}
By a simple calculation, we have that $\nabla \Phi(x_0,y_0)=\bf{0}$ if and only if $x_0 = 1 - \delta x_0 y_0$ and $y_0 = 1 + \delta x_0 y_0$. 
 Thus
\begin{equation}
x_0 + y_0 = 2, \qquad \text{and} \qquad y_0 - x_0 = 2 \delta x_0 y_0.
\end{equation}
Letting $r_0 = x_0 y_0$, we see that it satisfies the relation $r_0 = (1- \delta r_0)(1+\delta r_0) = 1 - \delta^2 r_0^2$.  Solving this explicitly, we see that $r_0$ is an even function of $\delta$, analytic for $|\delta| < 1/2$.  Note $r_0 = 1 - \delta^2 + O(\delta^4)$.
Then we have
\begin{equation}
\Phi(x_0, y_0) = \log(y_0/x_0) + x_0 - y_0 + \delta x_0 y_0 = \log\Big(\frac{1+\delta r_0}{1-\delta r_0}\Big) - \delta r_0,
\end{equation}
which is an odd function of $\delta$, with power series expansion of the form
$\Phi(x_0, y_0) =  \delta - \frac13 \delta^3 + \dots$.  Translating back to the original notation gives \eqref{eq:phiTaylor}.
\end{proof}

\begin{mylemma}
\label{lemma:Iexpansion2}
 Suppose that $\frac{U}{\epsilon N^2} = o(1)$.  
 \\
1.  Then $I(A,B,U,\epsilon,N) \ll_C N^{-C}$ with $C>0$ arbitrarily large, unless
\begin{equation}
\label{eq:ABsize2}
 |A| \asymp |B| \asymp \epsilon N, \quad A  < 0, \text{ and } B  > 0.
\end{equation}
2. Assuming \eqref{eq:ABsize2}, then
\begin{equation}
\label{eq:IasymptoticStationaryPhase2}
 I = \frac{1}{\epsilon} e^{i\phi(x_0, y_0)} W(\cdot) + O(N^2 U^{-C}),
\end{equation}
where $(x_0, y_0)$ is the unique solution to $\nabla \phi(x_0,y_0) = {\bf 0}$ and $W$ is $1$-inert in terms of any suppressed variables on which $w_N$ may depend.
\\
3.
Finally, $\phi(x_0, y_0)$ has the following Taylor expansion
\begin{equation}
\label{eq:phiTaylor2}
 \phi(x_0, y_0) = 
 \frac{AB}{\epsilon}
 \Big[
\sum_{j=0}^{J} c_j \Big(\frac{U \varepsilon}{AB}\Big)^{2j}
+ O\Big(\frac{U \varepsilon}{AB}\Big)^{2J+2}\Big)
\Big]
   + U \log\Big(\frac{-A}{B}\Big)
 ,
\end{equation}
with
certain absolute constants $c_j$.
\end{mylemma}
The condition $U = o(\epsilon N^2)$ means that the dominant phase in $\phi$ is $\epsilon xy$, and the phase $-U\log{x} + U\log{y}$ is a perturbation.

\begin{proof}
Considering the $x$-integral, Lemma \ref{lemma:exponentialintegral} shows that $I \ll N^{-C}$ unless
\begin{equation}
\Big|\frac{A}{\epsilon N} + \frac{y}{N} \Big| \ll \frac{U}{\epsilon N^2} = o(1).
\end{equation}
Since $1 \ll \frac{y}{N} \ll 1$ (with certain absolute implied constants), this means that $|A| \asymp |\epsilon| N$ with $A$ having the opposite sign of $\epsilon$ (i.e., $A <0$).  Similarly, considering the $y$-integral shows that $I$ is small unless $|B| \asymp \epsilon N$ with $B$ having the same sign as $\epsilon$ (i.e., $B > 0$).

Next we wish to apply Lemma \ref{lemma:exponentialintegralStatPhase}  to $I$. There is a minor technical issue from the fact that the second derivative with respect to $x$ (or $y$) of $\epsilon xy$  vanishes, even though this should be viewed as the dominant phase.  This issue may be circumvented by a simple change of variable to diagonalize this quadratic form.  Precisely, if we let $x = u+v$ and $y=u-v$, then
\begin{equation}
 \varphi(u,v):= \phi(u+v, u-v) = 
 \epsilon u^2 + \alpha u - \epsilon v^2 + \beta v 
 + U \log \Big(\frac{u-v}{u+v} \Big),
\end{equation}
for certain $\alpha,\beta$ whose values are immaterial.  Then a simple calculation gives
\begin{equation}
 \frac{\partial^2}{\partial u^2} \varphi(u,v) = 2 \epsilon  
 + U\Big(\frac{-1}{(u-v)^2} + \frac{1}{(u+v)^2}\Big)
 =
 2 \epsilon (1 + O(\epsilon^{-1} N^{-2} U)) \gg |\epsilon|.
\end{equation}
A similar calculation shows $|\frac{\partial^2}{\partial v^2} \varphi(u,v) | \gg |\epsilon|$.  Once we know that stationary phase can be applied after this linear change of variables, we can then revert back to the original variables $x,y$, giving
\begin{equation}
 I = \frac{1}{\epsilon} e^{i\phi(x_0, y_0)} W_T(\cdot) + O(N^{-C}),
\end{equation}
where $\nabla \phi(x_0, y_0) = {\bf 0}$.  We have
\begin{equation}
 \phi(Bx/\epsilon, -Ay/\epsilon) = 
\frac{-AB}{\epsilon}
\Phi(x,y)
   + U \log\Big(\frac{-A}{B}\Big),
\end{equation}
where
\begin{equation}
 \Phi(x,y) = xy - x -y + \delta \log(y/x),
 \quad
 \text{and}
 \quad
 \delta = \frac{U \epsilon}{AB} \asymp \frac{U}{\epsilon N^2} = o(1).
\end{equation}
A simple calculation shows $\nabla \Phi(x_0, y_0) = {\bf 0}$ if and only if
\begin{equation}
 x_0 = 1- \frac{\delta}{y_0}, \qquad y_0 = 1 + \frac{\delta}{x_0}.
\end{equation}
Solving these explicitly, we obtain
\begin{equation}
 x_0 = \frac{1-2 \delta + \sqrt{1+4\delta^2}}{2}, 
 \qquad
 y_0 = \frac{1+2 \delta + \sqrt{1+4\delta^2}}{2},
\end{equation}
and thus
\begin{equation}
 \Phi(x_0, y_0) = 
 - \frac{1 + \sqrt{1 + 4\delta^2}}{2} 
 -
 \delta \log\Big(\frac{1 + 2 \delta + \sqrt{1+4 \delta^2}}{1- 2 \delta + \sqrt{1+4\delta^2}}\Big)
 = -\sum_{j=0}^{\infty} c_j \delta^{j},
\end{equation}
which is analytic in $\delta$ for $|\delta| < 1/2$, and also even with respect to $\delta$. 
\end{proof}

 \begin{rema}
\normalfont Lemmas \ref{lemma:Iexpansion1} and \ref{lemma:Iexpansion2} have some close similarities.  In both cases, the stationary phase method may be applied, and the stationary point can be explicitly found by solving a quadratic equation.  In each case, only one of the two roots is relevant, and the other is outside the support of the test function.  We expect, but did not confirm rigorously, that when $U \asymp \epsilon N^2$,  which is a range that is not needed in this paper, then both roots of the quadratic equation are relevant.  This situation is more complicated because the two roots may approach each other in which case a cubic Taylor approximation to the phase function is more applicable (as with the Airy function, for instance).  
 \end{rema}

\section{Cleaning up some terms}
In this section we take the opportunity to deal with some ranges of parameters for which relatively easy methods suffice.  This will simplify our exposition for the more difficult cases.

With the aid of the analysis from Section \ref{section:Ibehavior} we can now treat some ranges of $c$.  %We begin with some discussion on the $v$-integral appearing in \eqref{eq:K+formula} and subsequent formulas. 

%Jutila and Motohashi \cite[(3.19)]{JM} showed that there is an asymptotic expansion for $K_+(x)$ via stationary phase.  The stationary point occurs at $v=v_0$ defined implicitly by
%\begin{equation}
%\sinh(v_0) = \frac{2T}{x},
%\end{equation}
%showing
%\begin{equation}
%\label{eq:JM-stat} 
%K_{+}(x) = x^{-1/2} e^{ix(\cosh(v_0) - 1) - 2i v_0 T} g_0(\blue{\Delta} v_0),
%\end{equation}
%plus a small error term,
%where $g_0(\Delta v_0)$ is 
%\blue{very small unless} $\frac{\Delta T}{x} \ll T^{\varepsilon}$, and $g_0(\Delta v_0)\ll 1$.  Note that if $\frac{T^2}{x} \ll T^{\varepsilon}$ then $K_{+}(x)$ essentially has no phase.  More precisely, it behaves like a $T^{\varepsilon}$-inert function.   It is this range which we treat below.

\begin{mylemma}
\label{lemma:smallCBound}
The contribution to $ \mathcal{S}(H_{+})$ from  $C \ll \frac{N^2}{T^2} T^{\varepsilon}$ is bounded by $\Delta T^{1+\varepsilon}$.
\end{mylemma}
\begin{proof}
Let $\mathcal{S}$ be the contribution to $ \mathcal{S}(H_{+})$ from  $C \ll \frac{N^2}{T^2} T^{\varepsilon}$. Since $x \asymp \frac{N^2}{C}$, the assumed upper bound on $C$ means $x \gg T^{2-\varepsilon}$, so that the conditions to apply Lemma \ref{lemma:K+cutoff} are in effect.  Applying Lemma \eqref{eq:K+cutoff} to \eqref{eq:Ikellcdef}, we deduce
\begin{equation}
I(k, \ell, c) = \int_{|v| \ll  x^{-1/2} T^{\varepsilon}} 
e^{- 2iTv} g(\Delta v) \eta(v) 
I(A, B, U, \epsilon(v), N)
dv + O(T^{-50}),
\end{equation}
with parameters as given in \eqref{eq:ABepsilondefs}.  Under the present assumptions, we have $\epsilon \ll \frac{v^2}{c} \ll \frac{T^{2\varepsilon}}{x c} \asymp \frac{T^{2\varepsilon}}{N^2}$.  Therefore, in the notation of \eqref{eq:ABepsilondefs}, we have $\epsilon N^2 \ll T^{2\varepsilon}$.

First consider the case where $U \gg T^{3\varepsilon}$.  In this case, $\epsilon N^2 = o(U)$, and so Lemma \ref{lemma:Iexpansion1} implies $I(A,B,U, \epsilon, N) \ll U^{-1} N^2$ and is very small unless $A \asymp B \asymp \frac{U}{N}$.  Translating notation, we may assume $|k| \asymp |\ell| \asymp \frac{CU}{N}$, and in particular, $k$ and $\ell$ are nonzero.  Integrating trivially over $v$, we deduce
\begin{equation}
\label{eq:IkellcboundInLemma}
I(k,\ell,c) \ll \frac{N C^{1/2} T^{\varepsilon}}{U} \Big(1 + \frac{|k| N}{CU}\Big)^{-100} \Big(1 + \frac{|\ell| N}{CU}\Big)^{-100}.
\end{equation}
Inserting this bound into \eqref{eq:SH+formula3}, we obtain 
\begin{multline}
|\mathcal{S}| \ll 
\frac{\Delta T T^{\varepsilon}}{UC}
\sum_{\substack{\nu, \gamma \geq 2, \thinspace \lambda \geq 4 \\ \min(\lambda, \nu) = \min(\lambda, \gamma)}} 
(2^{\nu}, 2^{\lambda-2-\delta})
\\
\sum_{k_o', \ell_o' \neq 0} 
\sum_{r_1, r_2}
\sum_{\substack{g_1|r_1^2 \\ g_2 | r_2^2}} g_1 g_2
\sum_{\substack{q^{\infty} \equiv 0 \shortmod{r_1} \\ 
	q \asymp \frac{C}{2^{\lambda} r_1^2 r_2^2}}}
	\Big(1 + \frac{|k_o' 2^{\nu} g_1 g_2| N}{CU}\Big)^{-100} \Big(1 + \frac{|\ell_o' 2^{\gamma} g_1 g_2| N}{CU}\Big)^{-100}.
\end{multline}
Estimating the sum trivially, and simplifying using $C \ll \frac{N^2}{T^2} T^{\varepsilon}$ and $N \ll N_{\text{max}} \ll U^{1/2} T^{1+\varepsilon}$, we deduce 
%\begin{equation}
%|\mathcal{S}'| \ll \frac{N}{U C} \frac{C}{r_1^* r_1^2 r_2^2} \Big( \frac{CU}{Ng_1 g_2}\Big)^2 T^{\varepsilon}.
%\end{equation}
%Inserting this into \eqref{eq:SintermsofS'}, we derive
\begin{equation}
 |\mathcal{S}| \ll \frac{\Delta T}{N} \frac{C^2 U}{N} T^{\varepsilon} \ll \frac{\Delta  U N^2}{T^3} T^{\varepsilon}  \ll \Delta T   \frac{U^2}{T^2} T^{\varepsilon},
\end{equation}
which is acceptable since $U \ll T$.

Next we indicate the changes needed to handle the case $U \ll T^{3\varepsilon}$.  
Integration by parts (Lemma \eqref{lemma:exponentialintegral}) shows that $I(A, B, U, \epsilon, N)$ is very small unless $A, B \ll \frac{T^{3 \varepsilon}}{N}$, equivalently, $|k|, |\ell| \ll \frac{C}{N} T^{3 \varepsilon}$.  Using $C \ll \frac{N^2}{T^2} T^{\varepsilon}$ and $N \ll N_{\text{max}} \ll T^{1+3 \varepsilon}$, this means that we only need to consider $k = \ell = 0$. A trivial bound implies $I(0, 0, c) \ll N C^{1/2} T^{\varepsilon}$.

%By integration by parts  , the same bound \eqref{eq:IkellcboundInLemma} holds in this case, but the difference from before is that we no longer have the condition $k, \ell \neq 0$.  Indeed, we may assume $|k|, |\ell| \ll \frac{CU}{N} T^{\varepsilon} \ll \frac{NU}{T^{2-\varepsilon}} \ll \frac{U^{3/2}}{T^{1-2\varepsilon}} \ll T^{-1+10\varepsilon}$.  Thus we only need to consider $k=\ell=0$.
  Using the final sentence of Lemma \ref{lemma:TabcCalculation}, we see that
 the contribution to $ \mathcal{S}$ from $k=\ell=0$ is bounded by
\begin{equation}
\frac{\Delta T}{N C^{3/2}}
\frac{N C^{1/2} T^{\varepsilon}}{U}
\sum_{r_2 \asymp C^{1/2}} C
\ll 
\frac{\Delta T}{U} T^{\varepsilon} C^{1/2}
\ll 
 \frac{\Delta N}{U} T^{\varepsilon}
 \ll \Delta T^{1+\varepsilon}. \qedhere
\end{equation}
\end{proof}

In light of Lemma \ref{lemma:smallCBound}, for the rest of the paper we can assume that
\begin{equation}
\label{eq:ClowerBound}
C \gg \frac{N^2}{T^2} T^{\varepsilon}.
\end{equation}
\begin{mylemma}
\label{lemma:vIsDyadicallyLocated}
Suppose \eqref{eq:ClowerBound} holds, and let
\begin{equation}
\label{eq:V0size}
V_0 = \frac{TC}{N^2}.
\end{equation}
Then with $x = \frac{4 \pi mn}{c} \asymp \frac{N^2}{C}$, we have
\begin{equation}
\label{eq:K+dyadicLocal}
K_{+}(x) = \int_{v \asymp V_0} 
e^{ix(\cosh(v) - 1) - 2iTv} g(\Delta v) \eta(v) dv + O((xT)^{-100}),
\end{equation}
where $\eta$ is a $1$-inert function supported on $v \asymp V_0$.
\end{mylemma}
Before proving Lemma \ref{lemma:vIsDyadicallyLocated}, we record a simple consequence of it which follows from inserting \eqref{eq:K+dyadicLocal} into \eqref{eq:Ikellcdef} (valid under the assumption \eqref{eq:ClowerBound}, which is in effect):
\begin{equation}
\label{eq:IkellAftervIsDyadicallyLocated}
I(k,\ell,c) = \int_{v \asymp V_0} 
e^{ix(\cosh(v) - 1) - 2iTv} g(\Delta v) \eta(v) I(A, B, U, \epsilon(v), N) dv + O(T^{-50}).
\end{equation}

\begin{proof}
In the definition of $K_{+}(x)$ given by \eqref{eq:K+formula}, we first apply a smooth dyadic partition of unity to the region $100 V_0 \leq |v| \ll \Delta^{-1} T^{\varepsilon} = o(1)$.  Consider a  piece of this partition, with say $Z \leq |v| \leq 2Z$.  We may apply Lemma \ref{lemma:exponentialintegral} with both $Y$ and $R$ taking the value $x Z^2$ (and $x \asymp \frac{N^2}{C}$).  Note $x Z^2 \gg \frac{N^2 V_0^2}{C} \gg T^{\varepsilon}$, so any such dyadic piece is very small.

Next we consider the portion of the integral with $|v| \leq \frac{V_0}{100}$.  The version of the integration by parts bound stated in Lemma \ref{lemma:exponentialintegral} is a simplified variant of \cite[Lemma 8.1]{BKY} (localized to a dyadic interval, etc.) which does not directly apply.  However, the more general \cite[Lemma 8.1]{BKY} can be used to show that this portion of the integral is also small.  The statement of \cite[Lemma 8.1]{BKY} contains a list of parameters (not to be confused with the notation from this paper) $(X, U, R, Y, Q)$ which in our present context take the values $(1, V_0, T, N^2/C, 1)$.  Lemma 8.1 from \cite{BKY} is sufficient to show the integral is very small provided $\frac{QR}{\sqrt{Y}} \rightarrow \infty$ and $RU \rightarrow \infty$.  Here $QR/\sqrt{Y}$ takes the form $\frac{T \sqrt{C}}{N} \gg T^{\varepsilon/2}$, and $RU = V_0 T \gg T^{\varepsilon}$, using the assumption \eqref{eq:ClowerBound}.
The remaining part of the integral is displayed in \eqref{eq:K+dyadicLocal}.

\end{proof}

\begin{mylemma}
\label{lemma:IkellcAsymptoticForUBig}
Suppose that the conditions of Theorem \ref{thm:mainthm} hold, as well as \eqref{eq:cUpperBound}.
Then
\begin{equation}
I(k,\ell,c) = \frac{N C^{1/2}}{U} \Big(\frac{k}{\ell}\Big)^{iU} 
\exp\Big(-\frac{2\pi i T^2 k \ell}{U^2 c} \Big)
W(\cdot) + O(T^{-100}),
\end{equation}
where $W$ is $1$-inert (in $k$, $\ell$, and $c$, as well as all suppressed variables), and supported on
\begin{equation}
\label{eq:kandellSizesWhenUisBig}
k \asymp \ell \asymp \frac{CU}{N}.
\end{equation}
\end{mylemma}
\begin{proof}
We begin by making some simple deductions from the conditions of Theorem \ref{thm:mainthm}.  
First we note that \eqref{eq:mainthmcondition} directly implies $U \Delta \geq T^{1+\delta}$.
Since \eqref{eq:cUpperBound} holds, we additionally deduce 
\begin{equation}
\label{eq:CisSmallish}
C \ll \frac{U N^2}{T^2} T^{-\delta},
\end{equation}
for some $\delta > 0$.
Another consequence of \eqref{eq:mainthmcondition} is that
\begin{equation}
\label{eq:UandDeltaAreBiggish}
\frac{T^3}{U^2 \Delta^3} \ll T^{-2\delta}.
\end{equation}
From the fact that $U \ll T$, we also deduce that (for some $\delta > 0$)
\begin{equation}
\label{eq:DeltaLowerBoundOneThirdPower}
\Delta \gg T^{1/3+\delta}.
\end{equation}

Now we pick up with \eqref{eq:IkellAftervIsDyadicallyLocated}.
Using \eqref{eq:V0size}, the condition \eqref{eq:CisSmallish} means that $\frac{\epsilon N^2}{U} \asymp \frac{V_0^2 N^2}{CU} \asymp \frac{T^2 C}{UN^2} \ll T^{-\delta}$, so that the conditions of Lemma \ref{lemma:Iexpansion1} are met.  This gives an asymptotic formula for the inner integral $I(A,B, U, \epsilon(v), N)$ for all $v \asymp V_0$. 
In particular, we deduce that $I(k, \ell, c)$ is very small unless \eqref{eq:kandellSizesWhenUisBig} holds, a condition that we henceforth assume is in place.
Note that by \eqref{eq:ABsize}
\begin{equation}
\frac{\epsilon U}{AB} = \frac{ (\cosh v -1) Uc}{\pi k \ell } \asymp \frac{UC V_0^2}{k \ell}
\asymp \frac{UC (T C/N^2)^2 }{(CU/N)^2}
 = \frac{T^2 C}{U N^2} \ll \frac{T}{U \Delta } T^{\varepsilon}
,  
\end{equation}
since $k \asymp \ell \asymp \frac{CU}{N}$, $v \asymp V_0$, and $C \ll \frac{N^2}{\Delta T} T^{\varepsilon}$ (recalling \eqref{eq:cUpperBound}).  Therefore,
\begin{equation}
\label{eq:use-of-assumption} U \Big(\frac{\epsilon U}{AB}\Big)^3 \ll  U \Big(\frac{T}{U \Delta}\Big)^3 T^{\varepsilon} \ll \frac{T^3}{U^2 \Delta^3} T^{\varepsilon} \ll T^{-\delta'},
\end{equation}
for some $\delta'>0$.  This calculation shows that in \eqref{eq:phiTaylor}, the terms with $j \geq 1$ can be absorbed into the inert weight function.  This is where we use the condition \eqref{eq:mainthmcondition} which can likely be relaxed to $U \Delta \gg T^{1+\delta}$, since this condition is sufficient to show that \eqref{eq:phiTaylor} is a good asymptotic expansion.
Therefore,
\begin{equation}
 I(k,\ell,c) = \frac{N^2}{U} \Big(\frac{k}{\ell}\Big)^{iU} \int_{v \asymp V_0} \exp\Big(-2iTv + i\frac{U^2 c(\cosh v -1)}{\pi k \ell}\Big) W(v, \cdot) dv,
\end{equation}
plus a small error term, where $W(v, \cdot)$ is $1$-inert with respect to $k, \ell, c$, and all other suppressed variables.  Next we can apply $\cosh(v) - 1 = v^2/2 + O(v^4)$ and absorb the $v^4$ terms into the inert weight function, using \eqref{eq:cUpperBound} and \eqref{eq:DeltaLowerBoundOneThirdPower} as follows:
\begin{equation}
\frac{U^2 C V_0^4}{k \ell} \asymp \frac{C^3 T^4}{N^6} \ll \frac{T}{\Delta^3} T^{3\varepsilon} \ll T^{-\delta'}.
\end{equation}
Finally, by stationary phase we obtain the desired estimate.
\end{proof}

Next we simplify our expression for $I(k,\ell,c)$ under the conditions of Theorem \ref{thm:CentralValueBound}, when $U$ is small.
\begin{mylemma}
\label{lemma:IkellcAsymptoticForUSmall}
Suppose that the conditions of Theorem \ref{thm:CentralValueBound} hold, as well as 
\eqref{eq:ClowerBound}.  Then $I(k,\ell,c)$ is very small unless
\begin{equation}
\label{eq:kandellSizesWhenUisTiny}
-k \asymp \ell \asymp \frac{C^2 T^2}{N^3},
\end{equation}
in which case
\begin{equation}
\label{eq:IkellcAsymptoticWhenUisTiny}
I(k,\ell, c) = \frac{N^4}{CT^2} (-k/\ell)^{iU}
e_c(-k \ell/12)
\int_{v \asymp V_0}  e^{-2ivT + \frac{2\pi i k \ell}{cv^2}} W(v,\cdot) dv
+ O(T^{-100}),
\end{equation}
for some function $W(v, \cdot)$ that is $1$-inert with respect to $k$, $\ell$, $c$, and all other suppressed variables.
\end{mylemma}
\begin{rema}
\normalfont Although it is possible to also evaluate the asymptotic of the $v$-integral in \eqref{eq:IkellcAsymptoticWhenUisTiny}, we prefer to save this step for later, in Section \ref{section:MellinInversion}.
\end{rema}
\begin{proof}
We again pick up with \eqref{eq:IkellAftervIsDyadicallyLocated} (recall also the definition \eqref{eq:Idef}), which takes the form
\begin{equation}
I(k,\ell,c) = \int_{v \asymp V_0} \eta(v) g(\Delta v) e^{-2ivT} I\Big(\frac{2 \pi k}{c}, \frac{2 \pi \ell}{c}, U, \epsilon, N\Big) dv,
\end{equation}
with $\epsilon=\epsilon(v) = 4\pi \frac{\cosh(v) - 1}{c} \asymp \frac{V_0^2}{C} \asymp \frac{C T^2}{N^4}$, for all $v \asymp V_0$.
Since \eqref{eq:ClowerBound} holds, this means that $\frac{U}{\epsilon N^2} \asymp \frac{U N^2}{T^2 C} \ll T^{-\varepsilon}$, so that the conditions of Lemma \ref{lemma:Iexpansion2} are met.  This directly implies that $I(k,\ell,c)$ is very small unless \eqref{eq:kandellSizesWhenUisTiny} holds.
Note that
\begin{equation}
\frac{AB}{\epsilon} = \frac{\pi k \ell}{c (\cosh v -1)},
\qquad
\text{and}
\qquad 
\Big|\frac{AB}{\epsilon} \Big| \Big(\frac{U \epsilon}{AB}\Big)^2 = \Big|\frac{U^2 \epsilon}{AB} \Big| \asymp \frac{U^2 N^2}{CT^2} \ll T^{-\varepsilon}.
\end{equation}
The latter calculation shows that the terms with $j \geq 1$ in \eqref{eq:phiTaylor2} may be absorbed into the inert weight function.  We thus conclude that
\begin{equation}
I(k,\ell, c) = \frac{N^4}{CT^2} (-k/\ell)^{iU}
\int_{v \asymp V_0}  e^{-2ivT +  \frac{\pi i k \ell}{c(\cosh v - 1)}} W(v,\cdot) dv
+ O(T^{-100}).
\end{equation}
Finally we observe the Taylor/Laurent approximation
\begin{equation}
\frac{1}{\cosh v -1} = \frac{2}{v^2} - \frac{1}{6} + O(v^2),
\end{equation}
and that
\begin{equation}
\frac{k \ell}{c} v^2 \asymp \frac{C^5 T^6}{N^{10}} \ll \frac{T}{\Delta^5} T^{\varepsilon} \ll T^{-\delta'},
\end{equation}
for some $\delta'>0$, where we have used $C \ll \frac{N^2}{\Delta T} T^{\varepsilon}$ from \eqref{eq:cUpperBound}.  This lets us absorb the lower-order terms in the Taylor expansion into the inert weight function.  Therefore, \eqref{eq:IkellcAsymptoticWhenUisTiny} holds.
\end{proof}

\section{Mellin inversion}
\label{section:MellinInversion}
We recall that we have the expression \eqref{eq:SH+formula3}, in which is contained a smooth (yet oscillatory) weight function of the form 
\begin{equation}
f(k,\ell,c) = e_c(k \ell/4) I(k,\ell,c).
\end{equation}
In the conditions of Theorem \ref{thm:mainthm}, we have that $I$ is given by Lemma \ref{lemma:IkellcAsymptoticForUBig}, while in the conditions of Theorem \ref{thm:CentralValueBound}, we have that $I$ is given by Lemma \ref{lemma:IkellcAsymptoticForUSmall}.  In both cases, the function $f$ is very small except when $k$ and $\ell$ are fixed into dyadic intervals.  We may therefore freely insert an inert weight function that enforces this condition.

First consider the setting relevant for Theorem \ref{thm:mainthm}.  The function $f$ has phase as given in Lemma \ref{lemma:IkellcAsymptoticForUBig}, modified to include $e_c(k \ell/4)$ which is  strictly smaller in size due to the assumption $U \leq (2-\delta)T$.  
We apply Lemma \ref{lemma:mellinIntegralOfAdditiveCharacter} to the phase function, and apply Mellin inversion to the inert part.  We therefore obtain
\begin{multline}
\label{eq:fkellcUBig}
f(k,\ell,c) = \frac{\Phi}{\sqrt{P}} \Big(\frac{2^{\nu} k_o'}{2^{\gamma} \ell_o'}\Big)^{iU}
\int_{-t \asymp P} \int \int \int
\Big(\frac{T^2 g_1^2 g_2^2 k_o' \ell_o'}{U^2 q r_1^2 r_2^2 2^{\lambda-\nu-\gamma}}\Big)^s \Big(1-\frac{U^2}{4T^2}\Big)^s v(t)
\widetilde{w}(u_1, u_2, u_3)
\\
\Big(\frac{C}{q r_1^2 r_2^2 2^{\lambda}}\Big)^{u_1}
\Big(\frac{K}{k_o' g_1 g_2 2^{\nu}}\Big)^{u_2}
\Big(\frac{K}{\ell_o' g_1 g_2 2^{\gamma}} \Big)^{u_3}
du_1 du_2 du_3 ds,
\end{multline}
plus a small error term, where $s=it$, and where
\begin{equation}
\label{eq:parameterDefsLargeU}
\Phi = \frac{N \sqrt{C}}{U},
\qquad
P = \frac{C T^2}{N^2},
\qquad
K = \frac{CU}{N}.
\end{equation}
By standard Mellin inversion of an inert function, the function $\widetilde{w}$ is entire and has rapid decay on any vertical line. However we do not specify the vertical contour in the integral above (and in several instances below). Also we have absorbed constants such as $\frac{1}{2\pi i}$ and the like into the weight functions.
We recall that $k = 2^{\nu} g_1 g_2 k_o'$, $\ell = 2^{\gamma} g_1 g_2 \ell_o'$, and $c = 2^{\lambda} q r_1^2 r_2^2$.  
We recall from Lemma \ref{lemma:mellinIntegralOfAdditiveCharacter} that $v(t)$ is supported on $-t \asymp P$, is $O(1)$, and has phase $e^{-it \log(|t|/e)}$.

We can also apply these steps to $I$ given by Lemma \ref{lemma:IkellcAsymptoticForUSmall}, which will have a similar structure but with an extra $v$-integral.  We obtain 
\begin{multline}
\label{eq:fkellcSmallU}
f(k,\ell,c) = \frac{\Phi_0}{\sqrt{P}} 
\int_{v \asymp V_0} e^{-2ivT} 
\Big(\frac{-2^{\nu} k_o'}{2^{\gamma} \ell_o'}\Big)^{iU}
\int_{-t \asymp P} \int \int \int
\Big(\frac{g_1^2 g_2^2 |k_o'| \ell_o'}{q r_1^2 r_2^2 2^{\lambda-\nu-\gamma}}\Big)^s \Big(\frac{1}{v^2} + \frac{1}{6} \Big)^s v(t)
\widetilde{w}(u_1, u_2, u_3)
\\
\Big(\frac{C}{q r_1^2 r_2^2 2^{\lambda}}\Big)^{u_1}
\Big(\frac{K}{|k_o'| g_1 g_2 2^{\nu}}\Big)^{u_2}
\Big(\frac{K}{\ell_o' g_1 g_2 2^{\gamma}} \Big)^{u_3}
du_1 du_2 du_3 ds dv,
\end{multline}
plus a small error term, where this time
\begin{equation}
\label{eq:parameterDefsSmallUprotoVersion}
\Phi_0 = \frac{N^4}{C T^2},
\qquad
P = \frac{C T^2}{N^2},
\qquad
K = \frac{C^2 T^2}{N^3},
\qquad
V_0 = \frac{CT}{N^2}.
\end{equation}
Here, $\widetilde{w}(u_1, u_2, u_3)$ is implicitly an inert function of $v$. It is the Mellin transform (in the suppressed variables, but not in $v$) of the function $W(v,\cdot)$  which was introduced in Lemma \ref{lemma:IkellcAsymptoticForUSmall}.

At this point, we finally asymptotically evaluate the $v$-integral.
We are considering 
\begin{equation}
\int_{v \asymp V_0} e^{-2ivT - 2s \log{v} + s\log(1+\frac{v^2}{6})} W(v,\cdot) dv,
\end{equation}
where we recall $s=it$, $-t \asymp P$.  We first observe that $s \log(1+\frac{v^2}{6}) = s v^2/6 + O(sv^4)$, and note 
\begin{equation}
|sv^4| \asymp P V_0^4 \ll \frac{T^{1+\varepsilon}}{\Delta^5} \ll T^{-\delta},
\end{equation}
by the assumption $\Delta \gg T^{1/5+\varepsilon}$.  Therefore, the term with $sv^4$ can be absorbed into the inert weight function at no cost.  We are therefore considering an oscillatory integral with phase $\phi(v) = -2vT - 2t \log{v} + t v^2/6$.  It is easy to see that $|\phi''(v)| \asymp \frac{P}{V_0^2}$ throughout the support of the test function, and that there exists a stationary point at $v_0$ satisfying
\begin{equation}
-2T - \frac{2t}{v_0} + \frac{t v_0}{3} = 0.
\end{equation}
We explicitly calculate
\begin{equation}
\label{eq:v0formula}
v_0 = \frac{2T - 2T \sqrt{1+ \frac{2 t^2}{3T^2}}}{2t/3} = \frac{-t}{T} + a' \frac{t^3}{T^3} + O\Big(\frac{P^5}{T^5}\Big),
\end{equation}
for some constant $a'$. We observe that $\frac{P^5}{T^4} \ll \frac{T^{1+\varepsilon}}{\Delta^5} \ll T^{-\delta}$, so quantities of this size (or smaller) may be safely discarded.
For later use, we note in passing that $\frac{P^2}{T^2} \ll \frac{T^{\varepsilon}}{\Delta^2} \ll T^{-\delta}$.
We conclude
\begin{equation}
\phi(v_0) = -2 t \log(|s|/T) + 2t + a \frac{t^3}{T^2} +O\Big(\frac{P^5}{T^4}\Big),
\end{equation}
for some new constant $a$.  Therefore,
\begin{equation}
\int_{v \asymp V_0} e^{-2ivT - 2 it \log{v} + it \log(1+\frac{v^2}{6})} w(v, \cdot) dv
= \frac{V_0}{\sqrt{P}} e^{-2it \log(\frac{|t|}{eT})} e^{ia \frac{t^3}{T^2}} W(\cdot),
\end{equation}
for some inert function $W$ and constant $a$.  Therefore, we deduce a formula for $f$ in the form
\begin{multline}
\label{eq:fkellcSmallU2}
f(k,\ell,c) = \frac{\Phi}{\sqrt{P}} 
\Big(\frac{-k_o'}{\ell_o'}\Big)^{iU}
\int_{-t \asymp P} \int \int \int
\Big(\frac{g_1^2 g_2^2 |k_o'| \ell_o'}{q r_1^2 r_2^2 2^{\lambda-\nu-\gamma}}\Big)^s v(t)
e^{-2it \log(\frac{|t|}{eT}) + ia\frac{t^3}{T^2}}
\widetilde{w}(u_1, u_2, u_3)
\\
\Big(\frac{C}{q r_1^2 r_2^2 2^{\lambda}}\Big)^{u_1}
\Big(\frac{K}{|k_o'| g_1 g_2 2^{\nu}}\Big)^{u_2}
\Big(\frac{K}{\ell_o' g_1 g_2 2^{\gamma}} \Big)^{u_3}
du_1 du_2 du_3 ds dv,
\end{multline}
where now
\begin{equation}
\label{eq:parameterDefsSmallU}
\Phi = \frac{N^4 V_0}{C T^2 P^{1/2}} = \frac{N^3}{C^{1/2} T^2},
\qquad
P = \frac{C T^2}{N^2},
\qquad
K = \frac{C^2 T^2}{N^3},
\qquad
V_0 = \frac{CT}{N^2}.
\end{equation}
This expression for $f(k,\ell,c)$ is similar enough to \eqref{eq:fkellcUBig} that we can proceed in parallel.  We mainly focus on the proof of Theorem \ref{thm:mainthm}.

Inserting \eqref{eq:fkellcUBig} into \eqref{eq:SH+formula4}, we obtain
\begin{multline}
\label{eq:SHFormulaInTermsOfZfunction}
\mathcal{S}_{\eta}(H_{+}) = 
\sum_C \frac{\Delta T}{NC^{3/2}} 
\frac{\Phi}{\sqrt{P}}
\int_{-t \asymp P} \int \int \int
\Big(\frac{T^2}{U^2} - \frac{1}{4}\Big)^s  v(t)
\widetilde{w}(u_1, u_2, u_3) 
\\
C^{u_1} K^{u_2 + u_3} 
Z(s,u_1,u_2,u_3)
du_1 du_2 du_3 ds,
\end{multline}
where $Z = Z_{\eta}$ is defined by
\begin{multline}
\label{eq:Zdef}
Z(s, u_1, u_2, u_3) =
\sum_{\substack{\nu, \gamma \geq 2, \thinspace \lambda \geq 4 \\ \min(\lambda, \nu) = \min(\lambda, \gamma) \\ \text{one of  \eqref{eq:alphabetalines} holds}}} 
\frac{(2^{\nu}, 2^{\lambda-2-\delta})}{2^{\lambda(u_1+s) + \nu(u_2-iU-s) + \gamma(u_3+iU-s)}}
\sum_{\substack{ (r_1 r_2, 2) = 1}}
\sum_{\substack{g_1 |r_1^2 \\ g_2 | r_2^2}} 
\\
\sumstar_{\substack{q: r_1 | q^{\infty} \\ (q,2r_2) = 1}}
\sum_{\substack{(k_o' \ell_o', 2)=1 \\ (k_o' \ell_o', \frac{r_2^2}{g_2}) = 1}}
\frac{\Big(\frac{q}{k_o' \ell_o'}\Big) \eta_1(k_o') \eta_2(\ell_o') \eta_3(q)
\prod_{\substack{p | r_2, \thinspace p \nmid \frac{r_2^2}{g_2}}} (1-p^{-1})}
	{(k_o')^{u_2-iU -s} (\ell_o')^{u_3+iU-s} q^{u_1+s} (r_1^2 r_2^2)^{u_1+s} (g_1 g_2)^{u_2+u_3-2s-1}}.
\end{multline}
We initially suppose that $\mathrm{Re}(s) = 0$ and $\mathrm{Re}(u_i) = 2$ for each $i$, securing absolute convergence of the sum.  An obvious modification, using \eqref{eq:fkellcSmallU2} in place of \eqref{eq:fkellcSmallU}, gives the corresponding formula for $U$ small, 
namely
\begin{multline}
\label{eq:SHFormulaInTermsOfZfunctionSmallU}
\mathcal{S}_{\eta}(H_{+}) = 
\sum_C \frac{\Delta T}{NC^{3/2}} 
\frac{\Phi}{\sqrt{P}}
\int_{-t \asymp P} \int \int \int
e^{-2it \log(\frac{|t|}{eT}) + ia\frac{t^3}{T^2}}  v(t)
\widetilde{w}(u_1, u_2, u_3) 
\\
C^{u_1} K^{u_2 + u_3} 
Z(s,u_1,u_2,u_3)
du_1 du_2 du_3 ds,
\end{multline}
where the parameters correspond with \eqref{eq:parameterDefsSmallU}, and the formula for $Z$ is slightly different (multiplied by $\eta_1(-1)$ to account for changing variables $k_o' \rightarrow -k$, with $k \geq 1$).

\section{Properties of the Dirichlet series $Z$}
In this section, we pause the development of $S_{\eta}(H_{+})$ and entirely focus on the Dirichlet series $Z$. 
\subsection{Initial factorization}
Throughout this section we assume that $\mathrm{Re}(s) = 0$.
For simplicity of notation only, we also take $\eta = (\eta_1, \eta_2, \eta_3)$ to be trivial, as the same proof works in the general case.

\begin{mydefi}
 Let $\mathcal{D}_0$ be the set of $(s,u_1, u_2, u_3) \in \mc^4$ with $\mathrm{Re}(s) = 0$, and
 \begin{equation}
\label{eq:ZdomainAbsConvergence}
\mathrm{Re}(u_1) > 1, \qquad \mathrm{Re}(u_2) > 1, \qquad \mathrm{Re}(u_3) >1.
\end{equation}
\end{mydefi}
It is easy to see that the multiple sum \eqref{eq:Zdef} defining $Z$ converges absolutely on $\mathcal{D}_0$.
We will work initially in $\mathcal{D}_0$, and progressively develop analytic properties (meromorphic continuation, bounds, etc.) to  larger regions.  
The largest domain in which we work is the following
\begin{mydefi}
 Let $\mathcal{D}_{\infty}$ be the set of $(s,u_1, u_2, u_3) \in \mc^4$ with $\mathrm{Re}(s) = 0$, and
 \begin{equation}
\label{eq:Dinfinity}
\mathrm{Re}(u_2) > 1/2, \qquad \mathrm{Re}(u_3) > 1/2, \qquad \mathrm{Re}(u_1) + \min(\mathrm{Re}(u_2), \mathrm{Re}(u_3)) >1.
\end{equation}
\end{mydefi}
Obviously, $\mathcal{D}_0 \subset \mathcal{D}_{\infty}$.

The following notation  will be useful throughout this section.  Suppose that $\mathcal{D}$ is a subset of $(s,u_1,u_2,u_3) \in \mc^{4}$ defined by $\mathrm{Re}(s) = 0$ and 
 by finitely many equations of the form $L(\mathrm{Re}(u_1), \mathrm{Re}(u_2), \mathrm{Re}(u_3)) > c$ where $c \in \mr$ and $L$ is linear with nonnegative coefficients.
For $\sigma > 0$, define $\mathcal{D}^{\sigma}$ by replacing each such equation by $L(\mathrm{Re}(u_1), \mathrm{Re}(u_2), \mathrm{Re}(u_3)) \geq c + \sigma$.  The nonnegativity condition means $\mathcal{D}^{\sigma} \subseteq \mathcal{D}$ for any $\sigma > 0$.

As a notational convenience, we write $k$ and $\ell$ instead of $k_0'$ and $\ell_0'$ in \eqref{eq:Zdef} (since there should be no danger of confusion with the original $k$ and $\ell$ variables).  In the domain $\mathcal{D}_0$, we may take the sums over $k$ and $\ell$ to the outside, giving 
\begin{equation}
\label{eq:Zsumwithkandellonoutside}
Z(s,u_1,u_2,u_3) = Z^{(2)}(s,u_1,u_2, u_3)
\sum_{(k \ell, 2) = 1}
\frac{Z_{k, \ell}(s,u_1,u_2,u_3)}{k^{u_2-iU-s} \ell^{u_3+iU-s}}
,
\end{equation}
where
\begin{equation}
Z_{k, \ell}(s,u_1,u_2,u_3)
=
\sum_{\substack{( r_1 r_2, 2) = 1 }}
\sum_{\substack{g_1 |r_1^2 \\ g_2 | r_2^2 \\ (\frac{r_2^2}{g_2}, k \ell ) = 1}} 
\sumstar_{\substack{q: r_1 | q^{\infty} \\ (q,2r_2) = 1}}
\frac{\Big(\frac{q}{k \ell}\Big) 
\prod_{\substack{p | r_2, \thinspace p \nmid \frac{r_2^2}{g_2}}} (1-p^{-1})}
	{q^{u_1+s} (r_1^2 r_2^2)^{u_1+s} (g_1 g_2)^{u_2+u_3-2s-1}},
\end{equation}
and
\begin{equation}
\label{eq:Z2def}
Z^{(2)}(s, u_1, u_2, u_3) = \sum_{\substack{\nu, \gamma \geq 2, \thinspace \lambda \geq 4 \\ \min(\lambda, \nu) = \min(\lambda, \gamma) \\ \text{one of  \eqref{eq:alphabetalines} holds}}} 
\frac{(2^{\nu}, 2^{\lambda-2-\delta})}{2^{\lambda(u_1+s) + \nu(u_2-iU-s) + \gamma(u_3+iU-s)}}.
\end{equation}
We first focus on properties of $Z_{k, \ell}$, and then turn to $Z^{(2)}$.

\subsection{Continuation of $Z_{k, \ell}$}
Note that $Z_{k,\ell}$ has an Euler product, say $Z_{k,\ell} = \prod_{p \neq 2} Z_{k,\ell}^{(p)}$.
It is convenient to define
\begin{equation}
\label{eq:alphabetadef}
\alpha = u_2 + u_3 - 2s - 1, 
\qquad
\beta = u_1 + s.
\end{equation}
Note that \eqref{eq:ZdomainAbsConvergence} implies $\mathrm{Re}(\alpha) > 1$ and $\mathrm{Re}(\beta) > 1$.  It is also convenient to observe that 
\begin{equation}
\label{eq:DinfinityalphabetaDeduction}
(s, u_1, u_2, u_3) \in \mathcal{D}_{\infty}
\Longrightarrow 
\mathrm{Re}(2 \alpha + 2 \beta) > 1 \quad \text{and}
\quad 
\mathrm{Re}(\alpha + 2 \beta) > 1.
\end{equation}

We evaluate $Z_{k,\ell}^{(p)}$ explicitly as follows.
\begin{mylemma}
Suppose that $\mathrm{Re}(\beta) > 0$ and $\mathrm{Re}(\alpha + \beta) > 0$.
For $p \nmid 2 k \ell$, we have
\begin{equation}
\label{eq:Zkellpformula}
Z_{k, \ell}^{(p)}(s,u_1,u_2,u_3) =  \frac{1+p^{-\alpha - 2\beta} - p^{-1-2\alpha-2\beta} + \chi(p) p^{-1-2\alpha-3\beta}}{(1- \chi(p) p^{-\beta})(1-p^{-2\alpha-2\beta})},
 \end{equation}
 where $\chi(n) = \chi_{k \ell}(n) = (\frac{n}{k \ell})$.
For $p | k \ell$, we have
 \begin{equation}
 \label{eq:ZkellpformulaRamified}
Z_{k, \ell}^{(p)}(s,u_1,u_2,u_3) =  \frac{1-p^{-1-2\alpha - 2\beta}}{1-p^{-2 \alpha - 2\beta}}.
 \end{equation}
 \end{mylemma}
\begin{proof}
For $(p, 2k \ell) = 1$, we have, using the convention $\infty \cdot 0 = 0$, 
\begin{equation}
Z^{(p)}(\alpha,\beta)
= 
\sum_{\min(r_1, r_2) = 0}
\sum_{\substack{0 \leq g_1 \leq 2r_1 \\ 0 \leq g_2 \leq 2 r_2}}
 (1-p^{-1})^{\delta_{g_2 = 2r_2 > 0}}
\sum_{\substack{ 0 \leq q \leq 1 \\ \infty \cdot q \geq r_1 \\ \min(q, r_2) = 0}} 
\frac{\chi(p^q)}{p^{\beta(q + 2 r_1 + 2r_2) + \alpha(g_1 + g_2)}}.
\end{equation}
We write this as $\sum_{r_2=0} + \sum_{r_2 \geq 1}$, where the latter terms force $q=r_1=0$.
We have 
\begin{equation}
\sum_{r_2 \geq 1} = \sum_{r_2=1}^{\infty} p^{-2 \beta r_2} 
\Big(
\sum_{0 \leq g_2 \leq 2r_2-1} p^{-\alpha g_2} 
+ (1-p^{-1}) p^{-2 \alpha r_2} \Big)
= 
\sum_{r_2=1}^{\infty} p^{-2 \beta r_2}
\Big(\frac{1-p^{-2\alpha r_2}}{1-p^{-\alpha}} + (1-p^{-1}) p^{-2 \alpha r_2}\Big).
\end{equation}
This evaluates as
\begin{equation}
(1-p^{-\alpha})^{-1} \Big( \frac{p^{-2\beta}}{1-p^{-2\beta}} - \frac{p^{-2\beta-2\alpha}}{1-p^{-2\alpha-2\beta}}\Big)
+ (1-p^{-1}) \frac{p^{-2\alpha - 2 \beta}}{1-p^{-2\alpha - 2\beta}},
\end{equation}
which simplifies as
\begin{equation}
p^{-2 \beta} \frac{1+p^{-\alpha}}{(1-p^{-2\beta})(1-p^{-2\alpha - 2\beta})} +
(1-p^{-1}) \frac{p^{-2\alpha - 2\beta} (1-p^{-2\beta})}{(1-p^{-2\alpha-2\beta})(1-p^{-2\beta})}.
\end{equation}
In turn this becomes
\begin{equation}
\frac{p^{-2 \beta}}{(1-p^{-2\alpha-2\beta})(1-p^{-2\beta})}
\Big[ 1 + p^{-\alpha} + (1-p^{-1}) p^{-2\alpha}(1-p^{-2\beta})\Big].
\end{equation}

Likewise, we compute
\begin{equation}
\sum_{r_2=0} = 
\sum_{r_1=0}^{\infty}
\sum_{\substack{0 \leq g_1 \leq 2r_1 }}
\sum_{\substack{ 0 \leq q \leq 1 \\ \infty \cdot q \geq r_1 }} 
\frac{\chi(p^q)}{p^{\beta(q + 2 r_1 ) + \alpha g_1}}
= 1 
+ \sum_{r_1=0}^{\infty} 
\sum_{\substack{0 \leq g_1 \leq 2r_1 }}
\frac{\chi(p)}{p^{\beta(1 + 2 r_1 ) + \alpha g_1}},
\end{equation}
by separating out  the cases $q=0$ and $q=1$.  We calculate this as
\begin{equation}
1 
+ \chi(p) p^{-\beta} \sum_{r_1=0}^{\infty} p^{-2\beta r_1}
\frac{1-p^{-\alpha(2r_1 + 1)}}{1-p^{-\alpha}},
\end{equation}
which can be expressed as
\begin{equation}
1 + \frac{\chi(p) p^{-\beta}}{1-p^{-\alpha}} \Big(\frac{1}{1-p^{-2\beta}} - \frac{p^{-\alpha}}{1-p^{-2\alpha-2\beta}}\Big)
=1 + \frac{\chi(p) p^{-\beta} (1+p^{-\alpha-2\beta})}{(1-p^{-2\beta})(1-p^{-2\alpha-2\beta})}.
\end{equation}
Putting the two calculations together, we obtain 
\begin{multline*}
Z^{(p)}(\alpha,\beta) = \\
\frac{(1-p^{-2\beta})(1-p^{-2\alpha-2\beta}) + 
\chi(p) p^{-\beta}(1+p^{-\alpha-2\beta}) + p^{-2\beta}(1+p^{-\alpha} + (1-p^{-1})(p^{-2\alpha} - p^{-2\alpha-2\beta}))
					}{(1-p^{-2\beta})(1-p^{-2\alpha-2\beta})}.
\end{multline*}
Distributing out the numerator and canceling like terms, we obtain
\begin{align}
\label{cancellation1} Z^{(p)}(\alpha,\beta) &= 
\frac{(1+ \chi(p) p^{-\beta})(1+p^{-\alpha-2\beta}) - p^{-1-2\alpha-2\beta}(1-p^{-2\beta})
					}{(1-p^{-2\beta})(1-p^{-2\alpha-2\beta})}.
%\label{cancellation2} &= \frac{1+p^{-\alpha - 2\beta} - p^{-1-2\alpha-2\beta} + \chi(p) p^{-1-2\alpha-3\beta}}{(1- \chi(p) p^{-\beta})(1-p^{-2\alpha-2\beta})}.
\end{align}
Simplifying gives \eqref{eq:Zkellpformula}.

Next, we need to consider the primes $p | k \ell$.
At such a prime we must have  $(q,p)=1$ (or else $(\frac{q}{k \ell})=0$) which implies $r_1 = 1$ and $g_2 = r_2^2$.  Thus
\begin{equation}
Z^{(p)}(s,u_1,u_2,u_3) = 
\sum_{r_2\ge 0   }
\frac{(1-p^{-1})^{\delta_{r_2>0}}}
	{ p^{r_2(2\beta + 2\alpha)} }
= 	 \frac{1-p^{-1-2\alpha - 2\beta}}{1-p^{-2 \alpha - 2\beta}}. \qedhere
\end{equation}
\end{proof}

Define the Dirichlet series
\begin{equation}
\label{eq:Ddef}
D(\alpha, \beta, \chi_{k \ell}) 
=
\sum_{(n,2) = 1} \frac{\mu^2(n)}{n^{\alpha + 2\beta}} \sum_{abc=n} \frac{\mu(b) \chi_{k \ell}(c)}{b^{1+\alpha} c^{1+\alpha+\beta}},
\end{equation}
which is absolutely convergent 
for $\mathrm{Re}(\alpha + 2 \beta) > 1$ and $\mathrm{Re}(\alpha + \beta) > 0$ (observe these conditions hold on $\mathcal{D}_{\infty}$, by \eqref{eq:DinfinityalphabetaDeduction}).  Note the Euler product formula
\begin{equation}
D(\alpha, \beta, \chi_{k \ell}) = \prod_{p \neq 2} (1 + p^{-\alpha - 2\beta}(1- p^{-1-\alpha} + \chi_{k \ell}(p) p^{-1-\alpha-\beta})).
\end{equation}
Putting together \eqref{eq:Zkellpformula} and \eqref{eq:ZkellpformulaRamified}, we deduce
(initially) in the region $\mathcal{D}_0$
\begin{equation}
\label{eq:ZkellEulerProductFormula}
Z_{k,\ell}(s, u_1, u_2, u_3) = 
L(\beta, \chi_{k \ell}) \frac{\zeta(2\alpha + 2\beta)}{(1-2^{-2\alpha - 2\beta})^{-1}}
D(\alpha, \beta, \chi_{k \ell}) 
(1- \chi_{k\ell}(2) 2^{-\beta}) 
\prod_{p | k \ell} a_p ,
\end{equation}
where
\begin{equation}
a_p = \frac{1 - p^{-1-2\alpha-2\beta}}{1 + p^{-\alpha - 2\beta} - p^{-1-2\alpha - 2\beta}}.
\end{equation}
Note that in $\mathcal{D}_{\infty}$, we have 
\begin{equation}
\label{eq:ZkellFiniteEulerProductEstimate}
a_p =  1 + O(p^{-1}).
\end{equation}

\begin{mylemma}
\label{lemma:Zproperties}
The series $Z_{k,\ell}(s,u_1, u_2, u_3)$ has meromorphic continuation to the domain $\mathcal{D}_{\infty}$.
In this region, $Z_{k,\ell}$ has a polar line only at $\beta = 1$ which occurs if and only if $\chi_{k \ell}$ is trivial. 
\end{mylemma}
\begin{proof}
This follows from \eqref{eq:ZkellEulerProductFormula}, using \eqref{eq:DinfinityalphabetaDeduction}.
\end{proof}

\begin{myremark}\label{remark-c} \normalfont Observe the nice simplification in the passage from \eqref{cancellation1} to \eqref{eq:Zkellpformula}, in which a factor of of $(1-p^{-2\beta})$ is canceled from the numerator and denominator. This reveals that there is no $\zeta(2\beta)^{-1}$ type factor in \eqref{eq:ZkellEulerProductFormula}, which would have infinitely many poles in the domain $\mathcal{D}_{\infty}$. 
\end{myremark}

%\begin{proof}
%The Dirichlet $L$-function $L(\beta, \chi_{k \ell})$ has meromorphic continuation to $\mc$ with a possible pole at $\beta =1$ that exists if and only if $\chi_{k \ell}$ is trivial.  Clearly $\zeta(2 \alpha + 2\beta)$ is analytic in \eqref{eq:ZkellDomain}, and it is easy to check the finite Euler products appearing in \eqref{eq:ZkellEulerProductFormula} are analytic as well.
%\end{proof}

\subsection{Evaluation of $Z^{(2)}$}
\label{section:Z2}
Recall that $Z^{(2)}$ has four cases, corresponding to \eqref{eq:alphabetalines}.
% may evaluate $Z^{(2)}$ explicitly, though it is a bit tedious to write out all four cases corresponding to \eqref{eq:alphabetalines}.
\begin{mylemma}
In cases (i)--(iii) of \eqref{eq:alphabetalines}, 
the function $Z^{(2)}$ initially defined by \eqref{eq:Z2def} in the region \eqref{eq:ZdomainAbsConvergence} extends to a bounded analytic function on $\mathcal{D}_{\infty}$.
% any domain of the form
% $\mathrm{Re}(u_2) \geq \sigma$, $\mathrm{Re}(u_3) \geq \sigma$, $\mathrm{Re}(\alpha + \beta) \geq \sigma$, $\mathrm{Re}(\alpha + 2\beta) \geq \sigma$
% with $\sigma > 0$.
\end{mylemma}
%Remark.  In particular, $Z^{(2)}$ extends analytically to $\mathcal{D}_1$.

\begin{proof}
This follows from brute force computation with geometric series.  
For case (i), we have
\begin{equation}
Z^{(2)} = \frac{(1-2^{-(u_2-iU-s)})^{-1} (1-2^{-(u_3+iU-s)})^{-1}}{2^{2+\delta} 2^{4(\alpha+\beta)} (1-2^{-\alpha-\beta})},
\end{equation}
which satisfies the claimed properties by inspection.  Cases (ii) and (iii) are easier, and give $Z^{(2)} = 2^{-1-\delta-3 \alpha - 4\beta} (1-2^{-\alpha-\beta})^{-1}$ and $Z^{(2)} = 2^{-\delta-2 \alpha - 4\beta} (1-2^{-\alpha-\beta})^{-1}$, respectively.
In case (ii), to see the boundedness on $\mathcal{D}_{\infty}$, note $2^{-3\alpha - 4\beta} = 2^{-2 \alpha - 2\beta} 2^{-\alpha - 2 \beta}$, and recall \eqref{eq:DinfinityalphabetaDeduction}.
\end{proof}
When $Z^{(2)}$ is given by case (iv), which recall restricts the summation to $\lambda \geq \nu +3$, it is convenient to split the sum into two pieces according to the size of $\lambda- \nu$.  
For any integer $L \geq 3$, write $Z^{(2)} = Z^{(2)}_{\leq L} + Z^{(2)}_{>L}$, where $Z^{(2)}_{\leq L}$ restricts to $\lambda- \nu \leq L$, and $Z^{(2)}_{>L}$ restricts to $\lambda- \nu > L$.
\begin{mylemma}
\label{lemma:Z2caseiv}
 In case (iv), $Z_{\leq L}^{(2)}$ extends to an analytic function 
on $\mathcal{D}_{\infty}$, wherein it satisfies the bound
 %on any domain of the form
% $\mathrm{Re}(\alpha + \beta) \geq \sigma$
% with $\sigma > 0$.  
%On this region, we have the bound
\begin{equation}
\label{eq:Z2initialsegmentBound}
|Z^{(2)}_{\leq L}| \ll L (2^{- L \beta} + 1). 
\end{equation}
The tail $Z_{> L}^{(2)}$ is analytic on $\mathcal{D}_0$
 %$\mathrm{Re}(\beta) \geq \sigma$ and $\mathrm{Re}(\alpha + \beta) \geq \sigma$, $\sigma > 0$,
wherein it satisfies the bound
 \begin{equation}
 \label{eq:Z2tailBound}
 |Z^{(2)}_{> L}| \ll 2^{-L \beta}.
 \end{equation}
\end{mylemma}
\begin{proof}
Since $\lambda \geq \nu + 3$, then $\min(\lambda, \nu) = \nu$, and the condition $\min(\lambda, \nu) = \min(\lambda, \gamma)$ means $\gamma = \nu$.  Therefore,
\begin{equation}
Z^{(2)}_{\leq L} =  \sum_{\nu \geq 2} \sum_{\nu + 3 \leq \lambda \leq \nu + L} \frac{2^{\nu}}{2^{\lambda \beta + \nu (\alpha + 1)}}
= 
\sum_{\nu \geq 2} \sum_{ 3 \leq \mu \leq L} \frac{1}{2^{(\nu + \mu) \beta + \nu \alpha}}
= \frac{2^{-2\alpha - 2 \beta}}{(1-2^{-\alpha-\beta})} \sum_{3 \leq \mu \leq L} 2^{-\mu \beta}.
\end{equation}
From this representation we easily read off its analytic continuation and the bound \eqref{eq:Z2initialsegmentBound}.
For the tail, we may modify the previous calculation to give 
\begin{equation}
\label{eq:Z2tailFormula}
Z^{(2)}_{> L}  
= \frac{2^{-2\alpha - 2 \beta}}{(1-2^{-\alpha-\beta})} \sum_{\mu \geq L+1} 2^{-\mu \beta}
= \frac{2^{-2\alpha - 2 \beta}}{(1-2^{-\alpha-\beta})} \frac{2^{-\beta(L+1)}}{(1-2^{-\beta})}
,
\end{equation}
from which we immediately read off the desired properties.
\end{proof}
\begin{rema}
\normalfont  Note that $Z^{(2)}_{>L}$
does not analytically continue to $\mathcal{D}_{\infty}$ since \eqref{eq:Z2tailFormula}  has
poles on the line $\mathrm{Re}(\beta) = 0$.  This explains the reason for splitting $Z^{(2)}$ into these two pieces.
\end{rema}

To unify the notation, in cases (i)--(iii), we define $Z_{>L}^{(2)} = 0$ and $Z_{\leq L}^{(2)} = Z^{(2)}$.
Corresponding to this decomposition of $Z^{(2)}$, we likewise write 
\begin{equation}
Z = Z_{\leq L} + Z_{>L}.
\end{equation}
With this definition, then the statement of Lemma \ref{lemma:Z2caseiv} holds in cases (i)--(iii) as well.  In this way we may henceforth unify the exposition for the four cases (i)--(iv).

\subsection{Continuation of $Z_{\leq L}$}
It is now useful to define another domain.
\begin{mydefi}
Let $\mathcal{D}_1$ be the set of $(s,u_1, u_2, u_3) \in \mc^4$ with $\mathrm{Re}(s) =0$, $\mathrm{Re}(u_2) > 1$, $\mathrm{Re}(u_3) > 1$, and satisfying
\begin{equation}
\label{eq:ZkellDomain}
\quad
\begin{cases}
\mathrm{Re}(u_1) +  \min(\mathrm{Re}(u_2), \mathrm{Re}(u_3)) > 3/2 \\
\mathrm{Re}(u_1) + 2 \min(\mathrm{Re}(u_2), \mathrm{Re}(u_3)) ) > 3.
\end{cases}
\end{equation}
\end{mydefi}
Note that $\mathcal{D}_0 \subset \mathcal{D}_1 \subset \mathcal{D}_{\infty}$.

\begin{mylemma}
\label{lemma:Zmeromorphic}
The series \eqref{eq:Zsumwithkandellonoutside} converges absolutely on $\mathcal{D}_1 \cap \{ \beta \neq 1 \}$ (and uniformly on compact subsets), which furnishes meromorphic continuation of the function $Z_{\leq L}$ to this domain.  Moreoever, the residue at $\beta = 1$ of $Z_{\leq L}$ is bounded for $\mathrm{Re}(u_2), \mathrm{Re}(u_3) > 1$.
\end{mylemma}
\begin{proof}
We return to \eqref{eq:Zsumwithkandellonoutside} and 
use the representation 
\eqref{eq:ZkellEulerProductFormula}, valid in $\mathcal{D}_0$.  The results from Section \ref{section:Z2} give
the analytic continuation of $Z^{(2)}_{\leq L}$ to $\mathcal{D}_{\infty}$ (and hence, $\mathcal{D}_1$).  Since $L(\beta, \chi_{k \ell})$ has a pole at $\beta = 1$ when $\chi_{k \ell}$ is trivial, we suppose $| \beta - 1| \geq \sigma > 0$, and will claim bounds with an implied constant that may depend on $\sigma$.
For $0 \leq \mathrm{Re}(\beta) = \mathrm{Re}(u_1) \leq 1$, we have the convexity bound $|L(\beta, \chi_{k \ell})| \ll_{\mathrm{Im}(\beta), \sigma, \varepsilon} (kl)^{ \frac{1-\mathrm{Re}(\beta)}{2} + \varepsilon}$ (with an implied constant depending at most polynomially on $\beta$).  One easily checks that \eqref{eq:Zsumwithkandellonoutside} converges absolutely for $\min(\mathrm{Re}(u_2), \mathrm{Re}(u_3)) + \frac{\mathrm{Re}(\beta)}{2} > \frac32$, which is one of the inequalities stated in \eqref{eq:ZkellDomain}.
Similarly, for $\mathrm{Re}(\beta) \leq 0$ we use the convexity bound $|L(\beta, \chi_{k \ell})| \ll (k \ell)^{\frac12 - \text{Re}(\beta) + \varepsilon}$ to see the absolute convergence for $\mathrm{Re}(u_1) + \min(\mathrm{Re}(u_2), \mathrm{Re}(u_3)) > 3/2$.
The uniform convergence on compact subsets is immediate, and so the meromorphic continuation follows.

Finally, to see the size of the residue, we simply note from \eqref{eq:ZkellEulerProductFormula} that $\text{Res}_{\beta = 1} Z_{k, \ell} \ll (k \ell)^{\varepsilon}$ for $\mathrm{Re}(u_2), \mathrm{Re}(u_3) \geq 1$.  In addition, the pole only exists if $k \ell$ is a square.
Moreover, $Z_{\leq L}^{(2)}$ is bounded at this point.  From \eqref{eq:Zsumwithkandellonoutside} we may then easily see the absolute convergence of the sum of these residues over $k, \ell$.
\end{proof}

\subsection{Functional equation}
Next we investigate how $Z_{k, \ell}$ and $Z_{\leq L}$ behave after an application of the functional equation of $L(\beta, \chi_{k \ell})$.  Suppose that $\chi_{k \ell}$ is induced by the primitive character $\chi^*$ of conductor $(k \ell)^*$.  We have
\begin{equation}
\Lambda(s,\chi^*) = ((k \ell)^*)^{s/2} \gamma(s) L(s, \chi^*) =  \Lambda(1-s, \chi^*),
\end{equation}
where $\gamma(s) = \pi^{-s/2} \Gamma(\frac{s+\kappa}{2})$, where $\kappa \in \{0, 1\}$ reflects the parity of $\chi$.  We therefore deduce the asymmetric form of the functional equation:
\begin{equation}
L(s,\chi_{k \ell}) =((k \ell)^*)^{\frac12 - s} \frac{\gamma(1-s)}{\gamma(s)}  L(1-s,\chi_{k \ell})
\prod_{p|k \ell} \frac{(1-\chi^*(p) p^{-s})}{(1-\chi^*(p) p^{s-1})}.
\end{equation}
\begin{mylemma}
\label{lemma:ZkellAfterFunctionalEquation}
In $\mathcal{D}_{\infty} \cap \{ \mathrm{Re}(\beta) < 0 \}$, we have
\begin{multline}
\label{eq:ZkellAfterFunctionalEquation}
Z_{k, \ell} = 
((k \ell)^*)^{\frac12 - \beta} \frac{\gamma(1-\beta)}{\gamma(\beta)}  
D(\alpha, \beta, \chi_{k \ell})
\frac{\zeta(2\alpha + 2\beta)}{(1-2^{-2\alpha - 2\beta})^{-1}} (1- 2^{-\beta} \chi_{k \ell}(2)) 
\\
\sum_{q=1}^{\infty} \frac{(\frac{q}{k \ell})}{q^{1-\beta}}
\prod_{p|k \ell} \frac{(1-\chi^*(p) p^{-\beta})}{(1-\chi^*(p) p^{\beta-1})}
\prod_{p | k \ell} a_p.
\end{multline}
\end{mylemma}
\begin{proof}
Lemma \ref{lemma:Zproperties} implies that the expression \eqref{eq:ZkellEulerProductFormula} for $Z_{k, \ell}$ is analytic on $\mathcal{D}_{\infty} \cap \{ \beta \neq 1\}$.
With the assumption $\mathrm{Re}(\beta) < 0$, we may apply the functional equation and express $L(1-\beta, \chi_{k \ell})$ in terms of its absolutely convergent Dirichlet series, which is \eqref{eq:ZkellAfterFunctionalEquation}.  
%Finally, we use \eqref{eq:ZkellFiniteEulerProductEstimate} to control the finite Euler product over $p|k \ell$ appearing in \eqref{eq:ZkellEulerProductFormula}.
\end{proof}

Having applied the functional equation to $Z_{k, \ell}$, the plan of action is to now insert this expression into the definition of $Z_{\leq L}$ and reverse the orders of summation, bringing $k$ and $\ell$ to the inside.  The outcome of this step is recorded with the following.
\begin{mylemma}
\label{lemma:ZlegLFormulalemma}
On $\mathcal{D}_1 \cap \{\mathrm{Re}(\beta) < 0 \}$, $Z_{\leq L}$ is a finite linear combination of absolutely convergent expressions of the form
\begin{equation}
\label{eq:ZleqLFormula}
Z^{(2)}_{\leq L} \frac{\gamma(1-\beta)}{\gamma(\beta)} \frac{\zeta(2\alpha + 2\beta)}{(1-2^{-2\alpha - 2\beta})^{-1}} \  \frac{(1 \pm 2^{-\beta})}{(1 \pm  2^{\beta-1})}
\sum_{(q,2) = 1} q^{\beta -1} \nu_1(q) A_q,
\end{equation}
with $A_q = A_q(s, u_1, u_2, u_3, U, \nu_2, \dots, \nu_6)$ defined by
\begin{multline}
\label{eq:Aqdef}
A_q = 
\sum_{(abc,2) = 1} \frac{\mu^2(abc) \nu_2(c) }{(abc)^{\alpha + 2\beta}} \frac{\mu(b)}{b^{1+\alpha} c^{1+\alpha+\beta}}
\\
\sum_{(k \ell, 2)=1} \frac{ (\frac{k \ell}{cq}) \nu_3(k) \nu_4(\ell)((k \ell)^*)^{\frac12 - \beta}}{k^{u_2-iU-s} \ell^{u_3+iU-s}}
\prod_{p|k \ell} \frac{(1-\chi_p((k\ell)^*) \nu_5(p) p^{-\beta})}{(1- \chi_p((k\ell)^* \nu_6(p)) p^{\beta-1})}
\prod_{p | k \ell} a_p,
\end{multline}
and where the $\nu_i$ run over Dirichlet characters modulo $8$.
\end{mylemma}
Observe that \eqref{eq:Aqdef} converges absolutely on $\mathcal{D}_1$.

\begin{proof}
Applying Lemma \ref{lemma:ZkellAfterFunctionalEquation} into \eqref{eq:Zsumwithkandellonoutside}, 
which is valid on $\mathcal{D}_1 \cap \{ \mathrm{Re}(\beta) < 0 \}$ by Lemma \ref{lemma:Zmeromorphic}, and applying the Dirichlet series expansion of $D(\alpha, \beta, \chi_{k \ell})$ given in \eqref{eq:Ddef}, 
we deduce 
\begin{multline}
Z_{\leq L}(s,u_1,u_2,u_3) = Z^{(2)}_{\leq L} \frac{\zeta(2\alpha + 2\beta)}{(1-2^{-2\alpha - 2\beta})^{-1}}
\sum_{(k \ell, 2) = 1}
\frac{((k \ell)^*)^{\frac12 - \beta}}{k^{u_2-iU-s} \ell^{u_3+iU-s}}
 \frac{\gamma(1-\beta)}{\gamma(\beta)}  
 \\
(1 - \chi_{k \ell}(2) 2^{-\beta})
\sum_{(abc,2) = 1} \frac{\mu^2(abc)}{(abc)^{\alpha + 2\beta}} \frac{\mu(b)}{b^{1+\alpha} c^{1+\alpha+\beta}}
\sum_{q=1}^{\infty} \frac{(\frac{qc}{k \ell})}{q^{1-\beta}} 
\prod_{p|k \ell} \frac{(1-\chi^*(p) p^{-\beta})}{(1-\chi^*(p) p^{\beta-1})}
\prod_{p | k \ell} a_p,
\end{multline}
where recall $a_p = 1 + O(p^{-1})$ on $\mathcal{D}_{\infty}$, and $\chi^* = \chi_{k \ell}^*$ is the primitive character induced by $\chi_{k \ell}(n) = (\frac{n}{k \ell})$ (so $\chi^*(n) = (\frac{n}{(k \ell)^*})$).

We next wish to focus on the sums over $k$ and $\ell$.  One small issue is that the parity of the character $\chi_{k \ell}$ (and hence the formula for $\gamma(s)$) may vary.  
However, the parity only depends on $k$ and $\ell$ modulo $8$.  
Also, $q$ may be even, but we can factor out the $2$-part of $q$ and directly evaluate its summation.
Likewise, we can apply quadratic reciprocity (again!) to give that $(\frac{qc}{k \ell})$ equals $(\frac{k \ell}{qc})$ times a function that depends only on $q,c, k,\ell$ modulo $4$.  
Similarly, we have that $\chi_{k \ell}^*(p)$ equals $\chi_p((k \ell)^*)$ up to a function modulo $4$.
We can then use multiplicative Fourier/Mellin decomposition modulo $8$ to express $Z_{\leq L}$ as a finite linear combination, with bounded coefficients, of sums of the form claimed in the statement of the lemma.
%\begin{multline}
%Z^{(2)}_{\leq L} \frac{\zeta(2\alpha + 2\beta)}{(1-2^{-2\alpha - 2\beta})^{-1}} \frac{\gamma(1-\beta)}{\gamma(\beta)}  
%\sum_{(k \ell, 2) = 1}
%\frac{((k \ell)^*)^{\frac12 - \beta} \nu_1(k) \nu_2(\ell)}{k^{u_2-iU-s} \ell^{u_3+iU-s}}
%\frac{(1 \pm 2^{-\beta})}{(1 \pm  2^{\beta-1})}
% \\
% \sum_{(abc,2) = 1} \frac{\mu^2(abc) \nu_3(c)}{(abc)^{\alpha + 2\beta}} \frac{\mu(b)}{b^{1+\alpha} c^{1+\alpha+\beta}}
%\sum_{(q,2)=1} \frac{(\frac{k \ell}{q}) \nu_4(q)}{q^{1-\beta}}
%\prod_{p|k \ell} \frac{(1-\chi_p((k\ell)^*) \nu_5(p) p^{-\beta})}{(1-\chi_p((k\ell)^*) \nu_6(p) p^{\beta-1})}
%\prod_{p | k \ell} a_p,
%\end{multline}
%for  Dirichlet characters $\nu_i$ modulo $8$. 
\end{proof}

Next we develop some of the analytic properties of $A_q$.  For notational convenience, we consider the case with all $\nu_i = 1$, as the general case is no more difficult.
We expand the Euler product over $p|k \ell$ involving $\chi^*$ into its Dirichlet series and reverse the orders of summation (taking $k, \ell$ to the inside), giving
\begin{equation}
\label{eq:Adef}
A_q = \sum_{(abc d e, 2)=1} \frac{\mu^2(abc) \mu(b) \mu(d)}{(abc)^{\alpha + 2\beta} b^{1 + \alpha} c^{1 + \alpha + \beta} d^{\beta} e^{1-\beta}} A_{q,c,d,e},
\end{equation}
where
\begin{equation}
A_{q,c,d,e} = \sum_{\substack{k \ell \equiv 0 \shortmod{d} \\ (k \ell)^{\infty} \equiv 0 \shortmod{e} \\ (k\ell, 2)=1}} \frac{ (\frac{k \ell}{cq})(\frac{(k \ell)^*}{de})((k \ell)^*)^{\frac12 - \beta}}{k^{u_2-iU-s} \ell^{u_3+iU-s}}
\prod_{p | k \ell} a_p.
\end{equation}
\begin{mylemma}
The function $A_{q,c,d,e}$ has meromorphic continuation to $\mathcal{D}_{\infty}$, in the form
\begin{equation}
\label{eq:AcMeromorphicContinuationFormula}
A_{q,c,d,e}  = L(u_1 + u_2 - iU - \tfrac12, \chi_{qcde}) L(u_1 + u_3 - iU - \tfrac12, \chi_{qcde}) C(\cdot)
\end{equation}
where $C = C_q(c,d,e, s,u_1,u_2,u_3, U)$ is a Dirichlet series analytic on $\mathcal{D}_{\infty}$ and satisfying the bound $C \ll ((de)')^{-2 \min \mathrm{Re} (u_2, u_3) + \varepsilon}$ on $\mathcal{D}_{\infty}$.  
\end{mylemma}
\begin{proof}
We initially work on $\mathcal{D}_1$ where the sum defining $A_{q,c,d,e}$ converges absolutely.
Now $A_{q,c,d,e}$ has an Euler product, taking the form $A_{q,c,d,e} = \prod_{(p,2)=1} A_{q,c,d,e}^{(p)}$, say, where
\begin{equation}
A_{q,c,d,e}^{(p)} = \sum_{\substack{k + \ell \geq v_p(d) \\ \infty \cdot (k+\ell) \geq v_p(e)} } \frac{(\frac{p^{k+\ell}}{cq})(\frac{(p^{k+\ell})^*}{de}) ((p^{k+\ell})^*)^{\frac12-\beta}}{p^{k(u_2-iU-s) + \ell(u_3 + iU - s)}} a_{p^{k+\ell}} ,
\end{equation}
where $v_p$ is the $p$-adic valuation, and
where we set $a_{p^{0}} = 1$ and $a_{p^j} = a_p$ for $j \geq 1$.

For the forthcoming estimates, we recall our convention from Section \ref{section:notation} that an expression of the form $O(p^{-s})$ should be interpreted to mean $O(p^{-\mathrm{Re}(s)})$.
If $p \nmid de$, then by separating the cases with $k+ \ell$ odd and $k+\ell$ even, we obtain
\begin{align*}
A_{q,c,d,e}^{(p)} &= 1 +  \Big(\frac{p}{qcde}\Big)
\Big[ 
\frac{1}{p^{u_1 + u_2 -iU   -\frac12}} 
+ 
\frac{1}{p^{u_1+u_3 +iU -\frac12}}
\Big] a_p + O(p^{-\min(2u_2,  2 u_3)})) \\
&= 
1 +  \Big(\frac{p}{qcde}\Big)
\Big[ 
\frac{1}{p^{u_1 + u_2 -iU   -\frac12}} 
+ 
\frac{1}{p^{u_1+u_3 +iU -\frac12}}
\Big] + O(p^{-\min(2u_2,  2 u_3)}) + O(\frac{p^{-1}}{p^{u_1 + \min(u_2, u_3) - \frac12}})
\\
& =
\frac{1 + O(p^{-\min(2u_2,  2 u_3)}) + O(\frac{p^{-1}}{p^{u_1 + \min(u_2, u_3) - \frac12}})
+ O(p^{-2(u_1 + \min(u_2, u_3) - \frac12)})}{ (1- \chi_{qcde}(p) p^{-u_1-u_2+iU+\frac12}) (1- \chi_{qcde}(p) p^{-u_1-u_3-iU+\frac12}) }.
\end{align*}
Note that on $\mathcal{D}_{\infty}^{\sigma}$, the $O$-term is of size $O(p^{-1-\sigma})$, 
and hence
\begin{equation}
\prod_{p \nmid de} A_{q,c,d,e}^{(p)} = L(u_1 + u_2 - iU - \tfrac12, \chi_{qcde}) L(u_1 + u_3 - iU - \tfrac12, \chi_{qcde}) B
\end{equation}
where $B = B(q,c,d,e,  s,u_1, u_2, u_3, U)$ is an Euler product that is absolutely convergent and bounded on $\mathcal{D}_{\infty}^{\sigma}$.

If $p | de$, then $(\frac{(p^{k+\ell})^*}{de})=0$ unless $(p^{k+\ell})^* = 1$, so we can assume that $k+\ell$ is even (and positive, hence $\geq 2$).  From such primes we obtain
$
A_{q,c,d,e}^{(p)} = O(p^{-\min(2u_2,  2 u_3)})
$, and hence
\begin{equation}
\prod_{p|de} A_{q,c,d,e}^{(p)} \ll ((de)')^{-2 \min \mathrm{Re} (u_2, u_3) + \varepsilon},
\end{equation}
where $(de)' = \prod_{p|de} p$.
Putting the estimates together, we deduce (initially in $\mathcal{D}_1$) the representation \eqref{eq:AcMeromorphicContinuationFormula}, where $C$ is analytic on $\mathcal{D}_{\infty}$.
Thus $A_{q,c,d,e}$ inherits the meromorphic continuation to $\mathcal{D}_{\infty}$ as well.
\end{proof}

\begin{mydefi}
 Let $\mathcal{D}_2$ be the set of $(s,u_1, u_2, u_3) \in \mc^4$ with $\mathrm{Re}(s) =0$, $\mathrm{Re}(u_2) > 1/2$, $\mathrm{Re}(u_3) > 1/2$, and satisfying
\begin{equation}
\mathrm{Re}(u_1) +  \min(\mathrm{Re}(u_2), \mathrm{Re}(u_3)) > 3/2.
\end{equation}
\end{mydefi}
One easily checks that $\mathcal{D}_1 \subset \mathcal{D}_2  \subset \mathcal{D}_{\infty}.$

\begin{mylemma}
\label{lemma:Aqproperties}
The function $A_q$ has meromorphic continuation to $\mathcal{D}_2 \cap \{\mathrm{Re}(u_1) < 1/2 \}$.
\end{mylemma}
\begin{proof}
%Next we examine the meromorphic properties of $A$ that may be inferred by the substitution of \eqref{eq:AcMeromorphicContinuationFormula} into \eqref{eq:Adef}.  
We (initially) work in the domain $\mathcal{D}_1$, where the absolute convergence is ensured. 
Substituting \eqref{eq:AcMeromorphicContinuationFormula} into \eqref{eq:Aqdef}, and letting $cde = r$, we obtain
\begin{equation}
\label{eq:Aformula}
A_q = 
\sum_{(r,2) = 1 } L(u_1 + u_2 - iU-\tfrac12, \chi_{qr})
L(u_1 + u_3 + iU-\tfrac12, \chi_{qr}) D(q,r)
,
\end{equation}
where
\begin{equation}
\label{eq:DformulaInTermsOfC}
D(q,r) = \sum_{\substack{(ab,2)=1 \\ cde = r}}
\frac{\mu^2(abc)  \mu(b) \mu(d)}{(abc)^{\alpha + 2\beta} b^{1 + \alpha} c^{1 + \alpha + \beta} d^{\beta} e^{1-\beta}} C_q(\cdot).
\end{equation}
We claim that $D(q,r)$ is analytic on $\mathcal{D}_{\infty} \cap \{ \mathrm{Re}(u_1) < 1/2 \}$ and therein satisfies the bound
\begin{equation}
\label{eq:Dsumbound}
D(q,r) \ll q^{\varepsilon} r^{\beta - 1 + \varepsilon} \prod_{p|r} p^{-2 u_1 - 2 \min(u_2, u_3)+1}.
\end{equation}
We now prove this claim.
 Recall \eqref{eq:DinfinityalphabetaDeduction}, which in particular immediately shows the
 absolute convergence in $\mathcal{D}_{\infty}$ of the free sum over $a,b$ in \eqref{eq:DformulaInTermsOfC}.  Hence
\begin{equation}
 |D(q,r)| 
\ll 
r^{\varepsilon}
\sum_{cde=r}
\frac{\mu^2(c) \mu^2(d) ((de)')^{-2 \min(u_2, u_3)}}{
c^{1 + 2\alpha + 3\beta} 
d^{\beta} 
e^{1-\beta}
}
= \sum_{cd | r} 
\frac{\mu^2(c) \mu^2(d) ((r/c)')^{-2 \min(u_2, u_3)}}{
c^{1 + 2\alpha + 3\beta} 
d^{\beta} 
(\frac{r}{cd})^{1-\beta}}
.
\end{equation}
One may now check \eqref{eq:Dsumbound} by brute force, prime-by-prime (by multiplicativity).

A consequence of \eqref{eq:Dsumbound} is that on $\mathcal{D}_{\infty} \cap \{\mathrm{Re}(u_1) < 1/2 \}$ we have the bound $D(q, p^k) \ll p^{k \varepsilon} p^{-1-\frac{k}{2}}$, for $p$ prime and $k \geq 1$, which extends multiplicatively.  Therefore,
 $\sum_{r} |D(q,r)| < \infty$ on $\mathcal{D}_{\infty} \cap \{ \mathrm{Re}(u_1) < 1/2 \}$.  The Dirichlet $L$-functions appearing in \eqref{eq:Aformula} are at most $O((qr)^{\varepsilon})$ on $\mathcal{D}_2$.  Therefore, \eqref{eq:Aformula} gives the meromorphic continuation of $A_q$ as stated in the lemma. 
\end{proof}

\begin{mylemma}
\label{lemma:ZpropertiesAfterFunctionalEquation}
On $\mathcal{D}_2 \cap \{ \mathrm{Re}(u_1) < 0 \}$, the function
 $Z_{\leq L}$ extends to a meromorphic function, on which it
is a finite linear combination of absolutely convergent sums of the form
\begin{multline}
\label{eq:Zdecomp}
Z_{\leq L}^{(2)} \frac{\gamma(1-\beta)}{\gamma(\beta)}
\sumstar_{ (r,2)=1} \sumstar_{(q,2)=1}  \frac{c_{q,r}}{q^{1-\beta}}
L(u_1 + u_2 - iU - \tfrac12, \chi_{qr} \nu) L(u_1 + u_3 - iU - \tfrac12, \chi_{qr} \nu'),
\end{multline}
where $\nu$, $\nu'$ are Dirichlet characters modulo $8$, and
$\sum^*$ means that the sum runs only over square-free integers.  Here $c_{q,r}$ is a Dirichlet series depending on $s, u_1, u_2, u_3, U$ that is analytic on $\mathcal{D}_{\infty} \cap \{ \mathrm{Re}(u_1) < 1/2 \}$, wherein it satisfies the bound
%\begin{equation}
%\label{eq:cqrdomain}
%\mathrm{Re}(u_1) + 2 \min(\mathrm{Re}(u_2), \mathrm{Re}(u_3) ) > 1,
%\end{equation}
%wherein it satisfies the bound
\begin{equation}
\label{eq:cqrbound}
c_{q,r} \ll  r^{-u_1 - 2 \min(u_2, u_3)} (qr)^{\varepsilon}.
\end{equation}
\end{mylemma}

\begin{proof}
%The plan is to derive the decomposition \eqref{eq:Zdecomp} initially on the domain $\mathcal{D}_1 \cap \{ \mathrm{Re}(\beta) < 0 \}$.  Then we claim that \eqref{eq:Zdecomp} in fact converges absolutely on $\mathcal{D}_2'\cap \{ \mathrm{Re}(\beta) < 0 \}$, seen as follows.  The Dirichlet $L$-functions are bounded on $\mathcal{D}_2'^{\sigma}$.  In addition, $\sum_{r} | c_{q,r} | < \infty$ on $\mathcal{D}_{\infty} \cap \{\mathrm{Re}(u_1) < 1/2 \}$ since $\mathrm{Re}(u_1) + 2 \min(\mathrm{Re}(u_2), \mathrm{Re}(u_3)) > \mathrm{Re}(u_1) + \min(\mathrm{Re}(u_2), \mathrm{Re}(u_3)) + \frac12 > 3/2$.
We work initially on the domain $\mathcal{D}_1 \cap \{ \mathrm{Re}(u_1) < 0 \}$ so that Lemma \ref{lemma:ZlegLFormulalemma} may be applied, giving \eqref{eq:ZleqLFormula}.   Now Lemma \ref{lemma:Aqproperties} may be invoked to give that $Z_{\leq L}$ is a linear combination of terms of the form
\begin{multline}
\label{eq:ZleqLDecomp3}
Z^{(2)}_{\leq L} \frac{\gamma(1-\beta)}{\gamma(\beta)} \frac{\zeta(2\alpha + 2\beta)}{(1-2^{-2\alpha - 2\beta})^{-1}} \  \frac{(1 \pm 2^{-\beta})}{(1 \pm  2^{\beta-1})}
\\
\sum_{(q,2) = 1} q^{\beta -1}
\sum_{(r,2) = 1 } L(u_1 + u_2 - iU-\tfrac12, \chi_{qr} \nu)
L(u_1 + u_3 + iU-\tfrac12, \chi_{qr} \nu') D(q,r),
\end{multline}
which converges absolutely on $\mathcal{D}_2 \cap \{ \mathrm{Re}(\beta) < 0 \}$. 
%since the Dirichlet $L$-functions are bounded by $O((qr)^{\varepsilon})$ on $\mathcal{D}_2' ^{\sigma}$, and since $\sum_{r} | D(q,r) | < \infty$ on $\mathcal{D}_{\infty} \cap \{\mathrm{Re}(u_1) < 1/2 \}$
 This gives the claimed meromorphic continuation of $Z_{\leq L}$.
 
 Next we show the claimed form \eqref{eq:Zdecomp}, which resembles closely the expression \eqref{eq:ZleqLDecomp3} except that we need to restrict $q$ and $r$ to be square-free.  Towards this end, replace $q$ by $q q_2^2$ and $r$ by $r r_2^2$ where the new $q$ and $r$ are square-free.  Note that $L(s, \chi_{q r q_2^2 r_2^2})$ agrees with $L(s, \chi_{qr})$ up to finitely many Euler factors that are bounded by $O((qr)^{\varepsilon})$ for $\mathrm{Re}(s) > 1/2$.  These finite Euler products can be incorporated into the definition of $D(q,r)$, which still satisfies \eqref{eq:Dsumbound} on $\mathcal{D}_{\infty} \cap \{ \mathrm{Re}(u_1) < 1/2 \}$.  Then we need to check the convergence in the sums over $q_2$ and $r_2$.  To this end, we first note simply that $\sum_{(q_2,2)=1} q_2^{2(\beta - 1)}=\zeta(2 - 2 \beta) (1- 2^{-2+2 \beta})$, which is analytic and bounded for $\mathrm{Re}(u_1) \leq 1/2 - \sigma$.  For $r_2$, we have from \eqref{eq:Dsumbound} that 
 \begin{equation}
 \sum_{r_2} |D(q, r r_2^2)| 
 \ll 
 \sum_{r_2} 
 q^{\varepsilon} (r r_2^2)^{\beta - 1 + \varepsilon} \prod_{p|r} p^{-2 u_1 - 2 \min(u_2, u_3)+1}
 \ll 
 (qr)^{\varepsilon} r^{- u_1 - 2 \min(u_2, u_3)}.
 \end{equation}
Finally, this gives the meromorphic continuation of $Z_{\leq L}$ to $\mathcal{D}_2 \cap \{\mathrm{Re}(u_1) < 0\}$ with the coefficients $c_{q,r}$ analytic on $\mathcal{D}_{\infty} \cap \{ \mathrm{Re}(u_1) < 1/2 \}$ and satisfying \eqref{eq:cqrbound}.
\end{proof}

\section{Completion of the proof of Theorem \ref{thm:mainthm}}
\label{section:proofofthm1}
Recall that the off-diagonal of $\sum_{T < t_j <T + \Delta} |L(\mathrm{sym}^2 u_j, 1/2+iU)|^2$ is a sum which we have been studying in dyadic intervals $n\asymp m\asymp N$ and $c\asymp C$ . Recall that $N \ll U^{1/2} T^{1+\varepsilon}$,  $C \ll \frac{N^2 T^{\varepsilon}}{\Delta T}$, and $C \gg \frac{N^2 T^{\varepsilon}}{T^2}$, originating from \eqref{eq:momentFirstFormula}, \eqref{eq:cUpperBound}, and \eqref{eq:ClowerBound}.
We also defined certain parameters $\Phi, P, K$ which can be found in \eqref{eq:parameterDefsLargeU}, but for convenience we recall here $\Phi = \frac{N \sqrt{C}}{U}$, $P = \frac{CT^2}{N^2}$, $K = \frac{CU}{N}$. Aided by the properties of $Z$ developed in the previous section, we are now ready to finish the proof of Theorem \ref{thm:mainthm}. 
We pick up from the expression \eqref{eq:SHFormulaInTermsOfZfunction}, where we begin with $\mathrm{Re}(u_1) = \mathrm{Re}(u_2) = \mathrm{Re}(u_3) = 2$.  Next we write $Z = Z_{\leq L} + Z_{>L}$, and choose $L$ so that $2^{L} \asymp C T^{\varepsilon}$.  To bound the contribution from $Z_{>L}$, we shift $u_1$ far to the right, and use the bound \eqref{eq:Z2tailBound}.  In terms of $u_1$, we get a bound of size $O((C/2^L)^{\mathrm{Re}(u_1)} \ll T^{-\varepsilon \mathrm{Re}(u_1)}$ which is negligible.  Next we focus on $Z_{\leq L}$.

We begin by shifting $u_1$ to the line $-\varepsilon$, which is allowed by Lemma \ref{lemma:Zmeromorphic}.  There is a pole of $Z_{\leq L}$ at $\beta = u_1 + s = 1$, with bounded residue.
 However, since $\mathrm{Im}(s) \asymp P$ and $P\gg T^\epsilon$, the weight function is very small at this height and the contribution from such poles are negligible.  Thus we obtain
\begin{multline}
\label{eq:SHafterFunctionalEquation}
\mathcal{S}(H_{+}) = 
\sum_C \frac{\Delta T}{NC^{3/2}} 
\frac{\Phi}{\sqrt{P}}
\int_{-t \asymp P} \int \int \int
\Big(\frac{T^2}{U^2} - \frac{1}{4}\Big)^s  v(t)
\widetilde{w}(u_1, u_2, u_3) 
C^{u_1} K^{u_2 + u_3} 
\frac{\gamma(1-u_1-s)}{\gamma(u_1+s)}
\\
Z^{(2)}_{\leq L} 
\sumstar_{r} \sumstar_{q} q^{u_1+s-1}c_{q,r}
L(u_1 + u_2 - iU - \tfrac12, \chi_{qr}) L(u_1 + u_3 - iU - \tfrac12, \chi_{qr})
du_1 du_2 du_3 ds,
\end{multline}
plus a small error term, as well as additional terms with the characters twisted modulo $8$. Since all our estimates hold verbatim for these additional twists, we suppress this from the notation.
Next we want to truncate the sums over $q$ and $r$.  To do so, we move $u_1$ far to the left, keeping $\mathrm{Re}(u_2) = \mathrm{Re}(u_3) = -\mathrm{Re}(u_1) + 100$.  Note that this remains in the domain $\mathcal{D}_{2}'$ and that $\mathrm{Re}(u_1) < 0$ so that the conditions of Lemma \ref{lemma:ZpropertiesAfterFunctionalEquation} remain in place to apply \eqref{eq:Zdecomp}.  Also, note
that the coefficients $c_{q,r}$ are $O(r^{-100})$ here.  Moreover, we observe by Stirling that
\begin{equation}
\Big|\frac{\gamma(1-u_1-s)}{\gamma(u_1+s)}\Big| \ll P^{\frac12 - \mathrm{Re}(u_1)}.
\end{equation}
In terms of the $u_1$-variable, the integrand in \eqref{eq:SHafterFunctionalEquation} is bounded by some fixed polynomial in $T$ times
\begin{equation}
 \Big(\frac{Cq}{PK^2} \Big)^{\mathrm{Re}(u_1)}.
\end{equation}
Therefore, we may truncate $q$ at $q \leq Q$ where
\begin{equation}
Q = \frac{P K^2}{C} T^{\varepsilon}.
\end{equation}
After enforcing this condition, 
and reversing the orders of summation (taking $r,q$ to the outside of the integrals), 
we shift the contours of integration so that $\mathrm{Re}(u_1) = 1/2-\varepsilon$ and $\mathrm{Re}(u_2) = \mathrm{Re}(u_3) = 1/2 + \varepsilon$; this is allowed by Lemma \ref{lemma:ZpropertiesAfterFunctionalEquation} as these contours shifts may be done in such a way that we remain in the domain $\mathcal{D}_{\infty} \cap \{ \mathrm{Re}(u_1) < 1/2 \}$ on which $c_{q,r}$ is analytic.  
Moreover, we observe from \eqref{eq:Z2initialsegmentBound} that $Z_{\leq L}^{(2)} \ll L \ll T^{\varepsilon}$ on this contour.
We then bound everything with absolute values, obtaining
\begin{multline}
\mathcal{S}(H_{+}) \ll T^{\varepsilon}
\max_C \frac{\Delta T}{NC^{3/2}} 
\frac{\Phi}{\sqrt{P}}
\ \int \int \int
\max_{\substack{x > 0 \\ q,r \ll Q}} \Big| \int_{-t \asymp P} x^{it} \frac{\gamma(1-u_1-it)}{\gamma(u_1+it)}  v(t) c_{q,r} dt \Big|
\\
|\widetilde{w}(u_1, u_2, u_3)|
C^{1/2} K
 \sumstar_{q \leq Q} q^{-1/2} 
|L(u_1 + u_2 - iU - \tfrac12, \chi_{q})|^2 
du_1 du_2 du_3.
\end{multline}
By Lemma \ref{lemma:SomeIntegralWithv}, keeping in mind that $c_{q,r}$ is given by a Dirichlet series, uniformly bounded in $t$ by \eqref{eq:cqrbound}, 
we have
\begin{equation}
\label{eq:gammaintegralLargeU}
\max_{x > 0} \Big| P^{-1/2} \int x^{it} \frac{\gamma(1-u_1-it)}{\gamma(u_1+it)} v(t)  c_{q,r} dt \Big| \ll 1.
\end{equation}
Applying \eqref{eq:HBsecondmomentwithweights}, we then obtain
\begin{equation}
\label{eq:SleqDBoundviaQuadraticLargeSieve}
\mathcal{S}(H_{+}) \ll T^{\varepsilon}
\max_C \frac{\Delta T}{NC^{}} 
\Phi
 K^{} 
(Q^{1/2} + U^{1/2}),
\qquad Q = \frac{PK^2}{C} T^{\varepsilon}.
\end{equation}
Therefore, we obtain 
\begin{equation}
\mathcal{S}(H_{+}) \ll T^{\varepsilon}
\max_C \frac{\Delta T}{NC^{}} 
\Phi
 K^{} 
 \Big(
 \frac{P^{1/2} K }{C^{1/2}} + U^{1/2}
 \Big)
 \ll 
 T^{\varepsilon}
 \max_C
 \frac{\Delta T}{NC^{}} 
\frac{N \sqrt{C}}{U}
 \frac{CU}{N}
 \Big(
 \frac{ T CU }{N^2} + U^{1/2}
 \Big).
\end{equation}
Using $C \ll \frac{N^2}{\Delta T} T^{\varepsilon}$, 
this simplifies as
\begin{equation}
\mathcal{S}(H_{+}) \ll
 T^{\varepsilon}
 \max_C
\frac{\Delta T}{N}
\sqrt{C}
 \Big(
 \frac{ T CU }{N^2} + U^{1/2}
 \Big)
 \ll  T^{\varepsilon}
 (\frac{T^{1/2} U }{\Delta^{1/2}}
  +   (\Delta T U)^{1/2} ).
\end{equation}
By \eqref{after-kuznetsov-2} and the remark following it, this implies
\begin{equation}
\label{eq:finalboundpresimplified}
 \sum_{T < t_j <T+\Delta} 
 |L(\mathrm{sym}^2 u_j, 1/2+iU)|^2
 \ll T^{\varepsilon}(\Delta T + \frac{T^{1/2} U }{\Delta^{1/2}})
 .
\end{equation}
We have $\Delta T\gg  \frac{T^{1/2} U }{\Delta^{1/2}}$ if and only if $\Delta\gg \frac{U^{2/3}}{T^{1/3}}$. This inequality holds because one of the conditions of Theorem \ref{thm:mainthm} requires $\Delta\gg \frac{T}{U^{2/3}}$, and $\frac{T}{U^{2/3}}\gg  \frac{U^{2/3}}{T^{1/3}}$ because $T\gg U$.

\section{Proving Theorem \ref{thm:CentralValueBound}}
For the proof of Theorem \ref{thm:CentralValueBound},  the parameters $\Phi, P, K$ are given in \eqref{eq:parameterDefsSmallU}, which for convenience we recall take the form $\Phi = \frac{N^3}{\sqrt{C} T^2}$, $P = \frac{C T^2}{N^2}$, $K = \frac{C^2 T^2}{N^3}$.  The bounds on $N$ and $C$ are the same as recollected in Section \ref{section:proofofthm1}.
The overall idea is to follow the same steps as in Section \ref{section:proofofthm1}, though picking up with \eqref{eq:SHFormulaInTermsOfZfunctionSmallU} instead of \eqref{eq:SHFormulaInTermsOfZfunction}.  The only structural difference between the two formulas is the additional phase of the form
\begin{equation}
e^{-2it \log(\frac{|t|}{eT}) + i a\frac{t^3}{T^2}}.
\end{equation}
Here the cubic term is of size $O(P T^{-\delta})$, as mentioned following \eqref{eq:v0formula}.
This only affects the argument in bounding \eqref{eq:gammaintegralLargeU}, but Lemma \ref{lemma:SomeIntegralWithv} is applicable (using the above remark that the cubic term is of lower-order) and gives the same bound with the above additional phase.
Referring to \eqref{eq:SleqDBoundviaQuadraticLargeSieve}, we thus obtain
\begin{equation}
\mathcal{S}(H_{+}) \ll T^{\varepsilon}
\max_C \frac{\Delta T}{NC^{}} 
\Phi
 K^{} 
 \Big(
 \frac{P^{1/2} K}{C^{1/2}} + U^{1/2}
 \Big)
 \ll 
 T^{\varepsilon}
 \max_C
 \frac{\Delta T}{NC^{}} 
\frac{N^3}{C^{1/2} T^2}
 \frac{C^2 T^2}{N^3}
 \Big(
 \frac{ T^3 C^2 }{N^4} + 1
 \Big).
\end{equation}
Using $C \ll \frac{N^2}{\Delta T} T^{\varepsilon}$, 
this simplifies as
\begin{equation}
\mathcal{S}(H_{+}) \ll  T^{\varepsilon}  \max_C \frac{\Delta T}{N} C^{1/2}
\Big(\frac{ T^3 C^2 }{N^4} + 1 \Big)
\ll  T^{\varepsilon}
\Big(\frac{T^{3/2} }{\Delta^{3/2}} 
+ (\Delta T)^{1/2}
\Big)
\end{equation}
Thus in all, by \eqref{after-kuznetsov-2} and the remark following it, we obtain
\begin{equation}
\sum_{T < t_j <T+\Delta} 
 |L(\mathrm{sym}^2 u_j, 1/2)|^2 \ll  T^{\varepsilon}
\Big(\Delta T 
+
\frac{T^{3/2}}{\Delta^{3/2}} 
\Big).
\end{equation}
The second term is smaller than the first term if and only if $\Delta \gg T^{1/5}$.


\begin{thebibliography}{99}


\bibitem[Ba]{Ba}
O. Balkanova,
\emph{The first moment of Maass form symmetric square $L$-functions.}
 Ramanujan J. 55 (2021), no. 2, 761–-781. 


\bibitem[BF]{BF}
O. Balkanova and D. Frolenkov, 
\emph{The mean value of symmetric square $L$-functions.} 
Algebra Number Theory 12 (2018), no. 1, 35--59. 

\bibitem[Bl1]{Bl}
V. Blomer,
\emph{On the central value of symmetric square L-functions. }
Math. Z. 260 (2008), no. 4, 755--777.

\bibitem[Bl2]{Bl2}
V. Blomer,
\emph{Sums of Hecke eigenvalues over values of quadratic polynomials.}
Int. Math. Res. Not. IMRN 2008 (2008), no. 16, 1--29.



\bibitem[BB]{blobut}
V. Blomer and J. Buttcane, 
\emph{On the subconvexity problem for L-functions on $GL(3)$.}
Ann. Sci. \'{E}c. Norm. Sup\'{e}r. (4) 53 (2020), no. 6, 1441--1500.





\bibitem[BH]{blohar}
V. Blomer and G. Harcos,
\emph{Hybrid bounds for twisted L-functions,}
J. Reine Angew. Math. 621 (2008), 53--79.

\bibitem[BHKM]{bhkm}
V. Blomer, P. Humphries, R. Khan, and M. Milinovich,
\emph{Motohashi's fourth moment identity for non-archimedean test functions and applications,}
Compos. Math. 156 (2020), no. 5, 1004--1038.


\bibitem[BKY]{BKY} V. Blomer, R. Khan, and M. Young, \emph{Distribution of mass of holomorphic cusp forms,} 
 Duke Math. J. 162 (2013), no. 14, 2609--2644.

\bibitem[Bo]{bour}
J. Bourgain,
\emph{Decoupling, exponential sums and the Riemann zeta function,}
J. Amer. Math. Soc. 30 (2017), no. 1, 205--224.

\bibitem[GR]{GR}
I. S. Gradshteyn and I. M. Ryzhik,
\emph{Table of integrals, series, and products,}
Translated from the Russian. Translation edited and with a preface by Alan Jeffrey and Daniel Zwillinger. Seventh edition. Elsevier/Academic Press, Amsterdam, 2007. 


\bibitem[HM]{HM}
G. Harcos and P. Michel, 
\emph{The subconvexity problem for Rankin-Selberg L-functions and equidistribution of Heegner points, II, }
Invent. Math. 163 (2006), no. 3, 581--655.

\bibitem[H-B]{HB}
D. R. Heath-Brown, 
\emph{A mean value estimate for real character sums,}
Acta Arith. 72 (1995), no. 3, 235--275. 

\bibitem[IK]{IK}
H. Iwaniec and E. Kowalski,
\emph{Analytic Number Theory}, volume~53 of {\em Colloquium Publications}. American Mathematical Society, Providence, RI, 2004.

\bibitem[IM]{iwamic}
H. Iwaniec and P. Michel
\emph{The second moment of the symmetric square L-functions,}
Ann. Acad. Sci. Fenn. Math. 26 (2001), no. 2, 465--482.

\bibitem[IS]{iwasar}
H. Iwaniec and P. Sarnak,
\emph{Perspectives on the analytic theory of L-functions,}
GAFA 2000 (Tel Aviv, 1999).
Geom. Funct. Anal. 2000, Special Volume, Part II, 705–741.

\bibitem[J]{J}
J. Jung,
\emph{Quantitative quantum ergodicity and the nodal domains of Hecke-Maass cusp forms,}
Comm. Math. Phys. 348 (2016), no. 2, 603–653.

\bibitem[JM]{JM}
M. Jutila and Y. Motohashi,
\emph{Uniform bound for Hecke L-functions,}
Acta Math. 195 (2005), 61--115.


\bibitem[K]{K}
R. Khan,
\emph{Non-vanishing of the symmetric square L-function at the central point,}
Proc. Lond. Math. Soc. (3) 100 (2010), no. 3, 736--762.

\bibitem[KD]{KD}
R. Khan and S. Das,
\emph{The third moment of symmetric square L-functions,}
Q. J. Math. 69 (2018), no. 3, 1063--1087.


\bibitem[KPY]{KPY}
E. K{\i}ral, I. Petrow, and M. Young, 
\emph{Oscillatory integrals with uniformity in parameters,}
J. Th\'{e}or. Nombres Bordeaux 31 (2019), no. 1, 145--159. 

\bibitem[KMS]{KMS}
S. Kumar, K. Mallesham, and S. K. Singh,
\emph{Sub-convexity bound for $GL(3)\times GL(2)$ $L$-functions: $GL(3)$-spectral aspect,}
preprint: arXiv:2006.07819.


\bibitem[La]{Lam}
J. W. C. Lam,
\emph{The second moment of the central values of the symmetric square L-functions,}
Ramanujan J. 38 (2015), no. 1, 129--145.

\bibitem[Li]{li}
X. Li,
\emph{Bounds for $GL(3)\times GL(2)$ L-functions and $GL(3)$ L-functions,}
Ann. of Math. (2) 173 (2011), no. 1, 301--336.

\bibitem[Lin]{lind}
E. Lindenstrauss,
\emph{Invariant measures and arithmetic quantum unique ergodicity,}
Ann. of Math. (2) 163 (2006), no. 1, 165--219.

\bibitem[MV]{micven}
P. Michel and A. Venkatesh,
\emph{The subconvexity problem for $GL_2$,}
Publ. Math. Inst. Hautes Études Sci. No. 111 (2010), 171--271.

\bibitem[M1]{mun1}
R. Munshi,
\emph{The circle method and bounds for L-functions—III: t-aspect subconvexity for $GL(3)$ L-functions,}
J. Amer. Math. Soc. 28 (2015), no. 4, 913--938.

\bibitem[M2]{mun2}
R. Munshi,
\emph{The circle method and bounds for L-functions—IV: Subconvexity for twists of $GL(3)$ L-functions,}
Ann. of Math. (2) 182 (2015), no. 2, 617--672.


\bibitem[N]{N}
P. D. Nelson,
\emph{Bounds for twisted symmetric square $L$-functions via half-integral weight periods,}
 Forum Math. Sigma 8 (2020), Paper No. e44, 21 pp. 

\bibitem[PY]{PY}
I. Petrow and M. P. Young,
\emph{The Weyl bound for Dirichlet L-functions of cube-free conductor,}
 Ann. of Math. (2) 192 (2020), no. 2, 437--486.

\bibitem[Sh]{sharma}
P. Sharma,
\emph{Subconvexity for {$GL(3)\times GL(2)$} {$L$}-functions in {$GL(3)$} spectral aspect,}
preprint: arXiv:2010.10153.

\bibitem[Soud]{soud}
D. Soudry,
\emph{On Langlands functoriality from classical groups to $GL_n$,} 
Astérisque, no. 298 (2005), 335--390.

\bibitem[Soun]{sound}
K. Soundararajan,
\emph{Invariant measures and arithmetic quantum unique ergodicity,}
Ann. of Math. (2) 172 (2010), no. 3, 1529--1538.



\bibitem[W]{W}
T. C. Watson,
\emph{Rankin triple products and quantum chaos,}
Thesis (Ph.D.)–Princeton University. 2002.

%\bibitem[Z]{Zagier}
%D. Zagier,
%\emph{Modular forms whose Fourier coefficients involve zeta-functions of quadratic fields.}
% Modular functions of one variable, VI (Proc. Second Internat. Conf., Univ. Bonn, Bonn, 1976), pp. 105--169. Lecture Notes in Math., Vol. 627, Springer, Berlin, 1977.


\end{thebibliography}
\end{document}